\documentclass{amsart}
\usepackage{amssymb,amsmath,amsthm,amscd,mathrsfs}
\usepackage{graphicx}
\usepackage[cmtip,all]{xy}

\numberwithin{equation}{section}
\newtheorem{teo}{Theorem}[section]
\newtheorem{pro}[teo]{Proposition}
\newtheorem{lem} [teo]{Lemma}
\newtheorem{cor}[teo]{Corollary}

\theoremstyle{definition}
\newtheorem{dfn}[teo]{Definition}

\newtheorem{notations}[teo]{Notations}

\theoremstyle{remark}
\newtheorem{rem}[teo]{Remark}

\newcommand{\bP}{\mathbb{P}}

\newcommand{\bQ}{\mathbb{Q}}
\newcommand{\bZ}{\mathbb{Z}}
\newcommand{\bC}{\mathbb{C}}

\newcommand{\calA}{\mathcal{A}}

\newcommand{\calF}{\mathcal{F}}

\newcommand{\calM}{\mathcal{M}}
\newcommand{\calO}{\mathcal{O}}
\newcommand{\calQ}{\mathcal{Q}}
\newcommand{\calP}{\mathcal{P}}

\newcommand{\calN}{\mathcal{N}}

\newcommand{\calH}{\mathcal{H}}
\newcommand{\calX}{\mathscr{X}}

\newcommand{\calD}{\mathcal{D}}
\newcommand{\calB}{\mathcal{B}}

\newcommand{\Sym}{\mathrm{Sym}}

\newcommand{\SL}{\mathrm{SL}}

\newcommand{\Proj}{\mathrm{Proj}}

\newcommand{\Pic}{\mathrm{Pic}}

\newcommand{\gquot}{/\!\!/}

\begin{document}
\bibliographystyle{amsplain}

\title[The Intermediate Jacobian Locus]{The moduli space of cubic threefolds via degenerations of the intermediate Jacobian}
\author[S. Casalaina-Martin]{Sebastian Casalaina-Martin}
\thanks{The first author was partially supported by NSF MSPRF grant DMS-0503228}
\address{Department of Mathematics, Harvard University, Cambridge, MA 02138, USA}
\email{casa@math.harvard.edu}

\author[R. Laza]{Radu Laza}
\address{Department of Mathematics, University of Michigan, Ann Arbor, MI 48109, USA}
\email{rlaza@umich.edu}

\begin{abstract} 
A well known result of Clemens and Griffiths  says that a smooth cubic threefold can be recovered from its intermediate Jacobian. In this paper we discuss the possible degenerations of these abelian varieties, and thus give a description of the compactification of the moduli space of cubic threefolds  obtained in  this way.  The relation between this compactification and those constructed in the work of Allcock-Carlson-Toledo and Looijenga-Swierstra is also considered, and is similar in spirit to the relation  between the various compactifications of the moduli spaces of low genus curves. 
\end{abstract}
\maketitle

\section{Introduction}
Cubic threefolds, that is smooth cubic hypersurfaces in $\mathbb P^4$, have been studied in many contexts since they were originally considered by the classical algebraic geometers such as Fano \cite{fano},  Todd \cite{todd} and Gherardelli \cite{gher}.  The relation between the geometry of a cubic threefold, and that of the canonical principal polarization of its intermediate Jacobian has been a central feature in the study of cubic threefolds since the seminal work of Clemens and Griffiths \cite{cg} in the nineteen seventies. Recently, several authors have turned their attention to the geometry of the  moduli space $\calM$ of  cubic threefolds, and in particular to its compactifications.   More specifically, Allcock \cite{A} and Yokoyama \cite{yokoyama} have constructed a compactification $\overline{\calM}$ of $\calM$ using geometric invariant theory.  With this as a starting point, Allcock-Carlson-Toledo \cite{act} and Looijenga-Swierstra \cite{ls} have shown that $\calM$ is the complement of an arithmetic arrangement $\calH$ of hyperplanes in a ten-dimensional ball quotient  $\calB/\Gamma$; the Baily-Borel compactification then gives another compactification $(\calB/\Gamma)^*$. The spaces $\overline{\calM}$ and $(\calB/\Gamma)^*$ are birational by construction, and the precise relation between them is one of the main results of \cite{act} and \cite{ls}.

From the perspective of Hodge theory (and in light of \cite{cg}), the construction in \cite{act,ls} of the compactification $(\calB/\Gamma)^*$ is rather indirect; it is based on an auxiliary construction  involving the period map for some special  cubic fourfolds. A more natural approach  would be to consider directly as in \cite{cg} the Hodge structure on the middle cohomology  $H^3(X)$ of a cubic threefold $X$ via  the  Griffiths intermediate Jacobian $J(X)$, a principally polarized abelian variety (ppav) of dimension five.  By the global Torelli theorem of Clemens and Griffiths \cite{cg}, the period map which assigns to the cubic $X$  its intermediate Jacobian $J(X)$ gives an embedding of the moduli space $\calM$ of cubic threefolds  into $\calA_5$, the moduli space of ppavs. The closure of the intermediate Jacobian locus $I$ (the image of $\calM$ in $\calA_5$) in the Satake compactification   of $\calA_5$ yields a natural projective compactification $\bar{I}$ of the moduli space of cubic threefolds.  One of the main problems in using this approach to study the moduli space is that it is difficult to characterize those abelian varieties which appear in the intermediate Jacobian locus $I$ and in the boundary $\partial I$.

A solution to the former problem, i.e. a characterization of the intermediate Jacobian locus $I$, was obtained by Friedman and the first author in \cite{cmf}:  the intermediate Jacobians of cubic threefolds are precisely those five dimensional ppavs whose theta divisor has a unique singular point, which is of multiplicity three.  It is then shown in \cite{casa} that the condition of a triple point on the theta divisor cuts in $\calA_5$ three ten dimensional irreducible components, one of which is $\bar{I}\cap \calA_5$.  
The aim of this paper is to complete the analysis of the locus $\bar{I}$, and then to relate it to the other compactifications of the moduli space of cubic threefolds mentioned above.

To state our results we introduce the following notation: $\bar{J}^h_5\subset \mathcal A_5^*$ is the locus of genus five hyperelliptic Jacobians,  $\bar{J}_4\subset \mathcal A_4^*\subset \partial\mathcal A_5^*$ is the locus of genus four Jacobians, $\theta_{\textnormal{null}}\subset \mathcal A_{4}^{*}$ is the closure of the locus of ppavs with a vanishing theta null, and  $K=\mathcal A_1\times (\bar{J}_4\cap  \theta_{\textnormal{null}})\subset \mathcal A_{5}^{*}$.

\begin{teo}\label{mainthm1} 
The boundary in $\mathcal  A_5^*$ of the locus of intermediate Jacobians of cubic threefolds consists of  three nine-dimensional irreducible components.  Namely, in the notation above, $\partial I= K \cup \bar{J}^h_5 \cup \bar{J}_4.$
\end{teo}

In particular, $\partial I\cap \mathcal A_{4}^{*}=\bar{J}_{4}$ and $(A,\Theta)\in \partial{I}\cap \mathcal A_5$ if and only if it is a Jacobian  of one of the following types: a hyperelliptic Jacobian, the product of two hyperelliptic Jacobians,  the product of an elliptic curve with the Jacobian of a smooth curve of genus four admitting a unique $g^1_3$, or the product of three or more Jacobians.

Intermediate Jacobians (of cubic threefolds) and their degenerations constitute an interesting class of abelian varieties in other contexts as well.   They occurr frequently in connection with the classical Schottky problem (e.g. Donagi \cite{donagis}), and also in connection with  a well known conjecture (cf. Debarre \cite{deb}) that Jacobians and intermediate Jacobians are the only  ppavs admitting subvarieties of minimal cohomology class. This  is  known in genus $4$, but open for higher genera. We note that contained in Theorem \ref{mainthm1} is the statement that this minimal cohomology class conjecture holds for the abelian varieties in $\bar{I}$. In other words, a degeneration of an intermediate Jacobian is a Jacobian.

Our main tool for the proof of Theorem \ref{mainthm1} is  Mumford's description of the intermediate Jacobian of a cubic threefold as the Prym variety associated to an \'etale double cover of a plane quintic. From this perspective, in order to understand $\bar{I}$, we can reduce the problem to the study of degenerations of Prym varieties associated to double covers of plane quintics. By work of  Beauville \cite{b} one has in principle a good understanding of this  situation. The main difficulty is that it is not clear  exactly which stable genus six curves arise as degenerations of smooth plane quintics (cf. Griffin \cite{griffin}, Hassett \cite{has} and  Hacking \cite{hacking}), and moreover it is not clear which admissible double covers arise in this way.  

We circumvent this problem by exploiting the close connection between  cubic threefolds and plane quintics: for a line on the cubic, projection from the line exhibits the cubic as a conic bundle over $\mathbb P^2$ with discriminant curve a plane quintic. Using GIT results for cubic threefolds and a careful choice of line on the cubic, we show that we can restrict our attention to a relatively small list of degenerations of plane quintics, for which the Beauville analysis can be carried  out. 
Essentially what makes our analysis possible is the fact that the fibers of the Prym map are  $2$-dimensional for plane quintics (e.g. Donagi-Smith \cite{ds}) so that for the purpose of understanding the intermediate Jacobian locus we do not  in fact have to identify all the degenerations of plane quintics, but only some ``generic'' ones.  The geometric analysis involved in passing from cubic threefolds to double covers of quintic curves is done in Section \ref{sectioncubics}. Most of the results there are adaptations of classical results to the singular case. The analysis of the degenerations of the intermediate Jacobians from the Prym perspective is then done in Sections \ref{secjac} and \ref{secpry}. This concludes  the proof of Theorem \ref{mainthm1}.

In the second part of our paper we discuss the relation between the compactification given by $\bar{I}$ and the other two compactifications mentioned above. We recall that Allcock--Carlson--Toledo \cite{act} and Looijenga--Swierstra \cite{ls} have constructed a space $\widehat{\calM}$ that dominates both $\overline{\calM}$ and $(\calB/\Gamma)^*$. With this as a starting point, it is not hard to produce a resolution of the natural rational map $\overline{\calM}\dashrightarrow  \bar{I}$. Namely, since the boundary of $\calM$ in $\widehat{\calM}$ is an arrangement of hyperplanes (see \cite{act} and \cite{ls}), we can take a stratified blow-up of the arrangement to obtain a normal crossing compactification $\widetilde{\calM}$ of the moduli space of cubic threefolds $\calM$, that dominates $\overline{\calM}$ and $(\calB/\Gamma)^*$ (see Looijenga \cite[\S2]{l1}). By the Borel extension theorem, $\widetilde{\calM}$ also automatically dominates  the closure of intermediate Jacobian locus $\bar{I}$, giving the diagram:
\begin{equation*}
\xymatrix{
&\widetilde{\calM}\ar@{->}[ld]\ar@{->}[rd]\ar@{->}[d]&\\ 
\overline{\calM}\ar@{-->}[r]&(\calB/\Gamma)^*\ar@{-->}[r]&\bar{I}
}
\end{equation*}
(for further details see \S\ref{secresolution}, esp. Diagram \ref{resdiag}).   

The extended period map  $\widetilde{\calM}\to  \bar{I}$ described in the previous paragraph carries with it some geometric information.  Roughly speaking, the limit intermediate Jacobian of a $1$-parameter degeneration of cubic threefolds $\calX\to \Delta$ is composed of two parts: an intrinsic part coming from the central fiber $X_0$ and a tail part determined by both the family as well as the type of singularities of  $X_0$.  This reflects the fact that the blow-up $\widetilde{\calM}\to (\calB/\Gamma)^*$ is stratified by the arithmetic type of the intersections of hyperplanes of  $\calH$, while the arithmetic type of the intersection corresponds in turn to the type of singularities of the cubic threefold. For precise statements we refer the reader to \S\ref{geomint} (esp. Theorem \ref{teodiv} and Table \ref{tablecon}).

A primary reason for our interest in the space $\widetilde{\mathcal M}$ is that it provides  an  essentially smooth space where   it is possible to use intersection theory  to make computations in the Picard group
$\Pic(\widetilde{\calM})_{\mathbb Q}$. In particular, we can make a ``numerical'' comparison between the spaces  $\overline{\calM}$, $(\calB/\Gamma)^*$, and $\bar{I}$ by comparing the classes of the line bundles   $L_0$, $L_1$, and $L_2$  obtained as pull-backs  to $\widetilde{\calM}$ of the canonical polarizations on  $\overline{\calM}$, $(\calB/\Gamma)^*$, and $\bar{I}$ respectively.  Note that the  line bundles $L_i$ are semi-ample; they determine the  morphisms from $\widetilde{\calM}$ to the respective compactifications which  can be described as  $\Proj(R(\widetilde{\calM},L_i))$, where $R(\widetilde{\calM},L_i):=\oplus_{k=1}^\infty H^0(\widetilde{\calM},L_i^{\otimes k})$ is the ring of sections. These computations are carried out in section \ref{sectcoef}. 
We have focused our attention on expressing these line bundles in terms of a basis of $\Pic(\widetilde{\calM})_{\mathbb Q}$ containing two geometrically important divisors: $D_{A_1}$, the divisor  parametrizing cubic threefolds with a node, and $H$, the divisor parametrizing ``hyperelliptic cubics'' (i.e. degenerations of cubic threefolds to the chordal cubic).  We obtain:

 \begin{teo}\label{mainthm3}
With notations as above, in $\Pic(\widetilde{\calM})_{\mathbb Q}$ modulo the exceptional divisors of $\widetilde{\calM}\to \widehat{\calM}$, we have $K_{\widetilde{\calM}}\equiv -\frac{7}{16}D_{A_1}+\frac{19}{8}H$ and:
\begin{eqnarray*}
L_0&\equiv&D_{A_1}+22 H \sim  K_{\widetilde{\calM}}+\frac{6}{11}D_{A_1}\\
L_1&\equiv&D_{A_1}+14 H\sim K_{\widetilde{\calM}}+\frac{17}{28}D_{A_1}\\
L_2&\equiv&\frac{1}{8}(D_{A_1}+2H)\sim K_{\widetilde{\calM}}+\frac{13}{8}D_{A_1}
\end{eqnarray*} 
where $D\sim D'$ indicates that $D\equiv \alpha D'$ for some $\alpha\in \mathbb Q$.
\end{teo}
 
We note that in principle it possible to compute the coefficients of the  polarizations $L_i$ for a complete basis of the Picard group; for instance, in \S\ref{secda2} we do this for another divisor $D_{A_2}$ corresponding to cubics with $A_2$ singularities.  We have restricted ourselves to $D_{A_1}$ and $H$ above for two reasons.  One is to avoid some of the more delicate computations; the other is that we think Theorem \ref{mainthm3} captures the essential aspect of the situation: as one moves from $\overline{\calM}$ to $\bar{I}$, the contribution of the divisor parametrizing singular cubics becomes more important (i.e. the polarization is $K_{\widetilde{\calM}}+\alpha D_{A_1}$ with $\alpha$ increasing). 

The computations of the coefficients are made in two ways. To compute $L_1$ (and also $L_0$) we use the explicit nature of the morphism $\widetilde{\calM}\to \widehat{\calM}$. This has a strong arithmetic flavor and relies on a good understanding of the action of the monodromy group $\Gamma$ on the ball $\calB$. For instance, the computation for $L_1$ is obtained by using the fact that the divisor of an automorphic form is an ample divisor on the Baily-Borel compactifications $(\calB/\Gamma)^*$. Combining the results of Allcock-Carlson-Toledo \cite{act} on $\Gamma$ with Borcherds' construction of automorphic forms, we immediately obtain the computation for $L_1$ (see \S\ref{sectcoef}). On the other hand, for $L_2$ we use particular families of cubic threefolds for which one can compute numerical invariants.  For Lefschetz pencils, the standard residue computations suffice.  For pencils containing cubics with other singularities, we convert the computation into one for families of  curves via the theory of Prym varieties.  This relies on some results of Bernstein \cite{bernstein} concerning the moduli of curves with level structure.

In section \ref{secthl}, we note that the situation described above for cubic threefolds is quite similar to the situation for low genus curves. For concreteness we restrict to the genus $g=3$ case. Namely,  for $\mathcal M_3$, the moduli space of genus $3$ curves,  there also exist  three special compactifications: a GIT compactification $\overline {\mathcal M}_3^{GIT}$, a ball quotient compactification $(\mathcal B_6/\Gamma_6)^*$ (see Kondo \cite{k1}), and  a compactification via abelian varieties $\bar{J}_3$. The birational geometry of the moduli space of genus $3$ curves is well understood by the work of Hyeon--Lee \cite{hl}.  In fact they show that these spaces are log canonical models; i.e. isomorphic to $\Proj$ of the ring of sections of $K_{\overline{\mathcal M}_3}+\alpha \delta$ for appropriate choices of $0\le \alpha\le 1$. Using the  same process described above for the moduli of cubic threefolds, one can construct a space $\widetilde{\mathcal M}_3$ resolving the rational maps $\overline{\calM}^{GIT}_3\dashrightarrow(\calB_6/\Gamma_6)^*\dashrightarrow \bar{J}_3$ (see Section \ref{secthl} and Diagram \ref{resdiagg3}).  The space $\widetilde{\mathcal M}_3$  is a blow-up of $\overline{\calM}_3$ and also provides a space on which to make computations.   As discussed in section \ref{secthl}, our computations in the proof of  Theorem \ref{mainthm3} are related to the computations of the coefficient $\alpha$ in Hyeon--Lee \cite{hl}. Essentially,  the difference between our work and that of \cite{hl} is a change of perspective.  Whereas they start from $\overline{\calM}_3$ and then try to find the log canonical models, we start with three natural spaces and then try to understand how they are related.  The idea is that in general (for example in the case of cubics) when there is no obvious space analogous to  $\overline{\calM}_3$, one may be able to use a (weak) substitute $\widetilde{\calM}$ on which it is still possible to do the computations.

The statement of Theorem \ref{mainthm1} is a result of some discussions with Samuel Grushevsky \cite{grup}.  While our proof is geometric and was developed
independently of \cite{grup}, he pointed out to us that a different proof can
be obtained via theta functions, and results from \cite{gru}.  The advantage of
the geometric proof given here is that it has direct consequences for the
resolution of the period map constructed in the second part of the paper;
for example it makes transparent the need to blow up the locus of cubics
with an $A_2$ singularity.

\subsection*{Acknowledgements} The second author would like to thank Daniel Allcock for explaining to him several aspects of the monodromy group for cubic threefolds, and Shigeyuki Kondo for informing him about his work on automorphic forms for genus $3$ curves.  The first author would like to thank Samuel Grushevsky for explaining some related work on theta functions pertaining to degenerations of intermediate Jacobians of cubic threefolds.  We are also grateful to  our advisor Robert Friedman for raising our interest in cubic threefolds and their moduli.

\subsection{Notations and Conventions}  We work over the complex numbers.  $\calM$ denotes throughout the coarse moduli space of smooth cubic threefolds, which is considered as an open subset of the GIT compactification $\overline{\calM}$ (see \cite{A,yokoyama}). Appropriate period maps  embed  $\calM$ into a ball quotient $\calB/\Gamma$ (see \cite{act,ls}) and into the moduli of abelian varieties $\calA_5$ (see \cite{cg}) respectively. We denote by $(\calB/\Gamma)^*$  and $\bar{I}$ respectively the compactifications into the corresponding Satake-Baily-Borel compactifications. 

We are interested in  hypersurface singularities of analytic type type $A_k$ and $D_n$, i.e. given by equations $x_1^{k+1}+x_2^2+\dots x_n^2$ and $x_1^{n-1}+x_1 x_2^2+x_3^2\dots +x_n^2$ respectively. All singularities are considered up to stable equivalence (i.e. up to adding squares of new variables). In particular, it makes sense to say that a threefold singularity has the same analytic type as a curve singularity.
                                                                                    
\section{The moduli space of cubic threefolds}\label{secgit} The moduli space of cubic threefolds was analyzed recently in a series of papers \cite{A,yokoyama,act,ls}. The most accessible description, and the starting point of more detailed investigations, is to realize the moduli space of smooth cubic threefolds as an open subset of a GIT quotient.  This construction automatically gives a projective compactification $\overline{\calM}$ with some geometric meaning. The relevant results are summarized in the following theorem.

\begin{teo}[\cite{A,yokoyama}]\label{thmgit} 
In the notation above, 
\begin{itemize}
\item[i)] A cubic threefold is GIT stable if and only if it has at worst $A_4$ singularities (\cite[Thm. 1.1]{A}).
\item[ii)]  A semistable cubic threefold has at worst $A_k$ or $D_4$ singularities or is the chordal cubic (\cite[Thm. 1.3]{A}).
\item[iii)] The boundary of the stable locus consists of a rational curve $T_{\beta}$ and a special point $\Delta$  (cf. \cite[Thm. 1.2]{A}). 
\end{itemize}
\end{teo}

The curve $T_{\beta}$ mentioned in the theorem parametrizes the threefolds of equation $F_{A,B}=A x_2^3+x_0x_3^2+x_1^2x_4-x_0x_2x_4+Bx_1x_2x_3$, where $\beta=\frac{4A}{B^2}$. If $\beta\neq 0, 1$,   the threefold of equation $F_{A,B}$ has precisely $2$ singularities of type $A_5$. In the two degenerate cases, the threefold either acquires an additional node, or it becomes the chordal cubic (i.e. the secant variety of the rational normal curve of degree $4$). The special point of type $\Delta$ corresponds to the orbit of the threefold given by $F_{\Delta}=x_0x_1x_2+x_3^3+x_4^3$. This threefold has $3$ singularities of type $D_4$.

Our aim  is to understand the degenerations of intermediate Jacobians of cubic threefolds. For this purposes, it is enough to understand the $1$-parameter degenerations of smooth cubic threefolds. Since any such degeneration can be filled in with a semi-stable cubic threefold having a closed orbit (\cite[Prop. 2.1]{shah}), it follows that we only have to consider a rather restricted list of possibilities for the central fiber of the degeneration. We distinguish two cases: either the central fiber has isolated singularities or it is the chordal cubic (singular along the rational normal curve).  The degenerations to the chordal cubic were analyzed by Collino \cite{col} and  will be considered separately in this paper.  Thus, we will be mostly concerned with $1$-parameter degenerations to cubics with isolated ``allowable'' singularities.

\begin{dfn} We say that a hypersurface singularity is {\it allowable} if  it is either of type $A_k$ for some $k\le 5$ or of type $D_4$.
\end{dfn}

Another description of the moduli space of cubic threefolds $\calM$ was obtained by Allcock--Carlson--Toledo \cite{act} and Looijenga--Swierstra \cite{ls}. Using the periods of cubic fourfolds obtained as the triple cover of $\mathbb P^4$ branched along a cubic threefold, and Voisin's  global Torelli theorem for cubic fourfolds, it is shown that 
the moduli space of cubic threefolds is an open subset of a locally symmetric variety $\calB/\Gamma$, the quotient of a $10$-dimensional complex ball by an arithmetic group. This is then naturally compactified to a projective variety $(\calB/\Gamma)^*$. The varieties $\overline{\calM}$ and $(\calB/\Gamma)^*$ are birational, being isomorphic over the open subset $\calM$. The main result of \cite{act,ls} is to make explicit the nature of this birational map. Specifically, one shows that the indeterminacy locus of the extended period map $\overline{\calP}_c:\overline{\calM}\dashrightarrow (\calB/\Gamma)^*$ is a single point, the boundary point corresponding to the chordal cubic. The blow-up of this point in a specific way (see \cite[\S4]{act}) gives a space $\widehat{\calM}$  over which the period map extends to a morphism. The results is the following diagram:
\begin{equation}\label{blowup}
\xymatrix{
&\widehat{\calM}\ar@{->}[ld]_{\epsilon_1}\ar@{->}[rd]^{\widehat{\calP}_c}&\\ 
\overline{\calM}\ar@{-->}[rr]^{\overline{\calP}_c}&&(\calB/\Gamma)^*
}
\end{equation}

\begin{teo}[\cite{act,ls}]\label{teoresolution0} With notation as above. The blow-up of the GIT quotient $\overline{\calM}$ in the point corresponding to the orbit of the chordal cubic resolves the birational map $\overline{\calP}_c$. Additionally,
\begin{itemize}
\item[(1)] $\overline{\calP}_c$ is an isomorphism over $\overline{\calM}\setminus T_\beta$ with image the complement of an irreducible divisor $H_{c}$;
\item[(2)]  the image of $\calM$ is the complement of the union of the divisors $H_{c}$ and  $H_{\Delta}$, and ${\overline{\calP}_c}_*\Sigma=H_{\Delta}$, where $\Sigma$ is the discriminant divisor in $\overline{\calM}$; 
\item[(3)] the morphism $\widehat{\calP}_c$ contracts the strict transform of the curve $T_{\beta}$; 
\item[(4)] via the small morphism $\widehat{\calP}_c$, the Weil divisors $H_{c}$ and $H_{\Delta}$ on $(\calB/\Gamma)^*$ become  $\bQ$-Cartier on $\widehat{\calM}$.
\end{itemize}
\end{teo}

There are two reasons why such a description is very useful. First, from a geometric point of view, the description $(\calB/\Gamma)^*$ takes into account the hyperelliptic degenerations of the intermediate Jacobians, essentially incorporating the analysis of Collino \cite{col}. Secondly, and more importantly for us, it is the explicit nature of the space $(\calB/\Gamma)^*$ and of the divisors $H_c$ and $H_\Delta$. We call $H_c$ and $H_\Delta$ the {\it hyperelliptic divisor} and the {\it discriminant divisor} respectively (N.B. the hyperelliptic divisor is called the chordal divisor in \cite{act}). For further reference we recall the following facts (cf. \cite{act} and \cite{ls}) about the divisors $H_c$ and $H_\Delta$:
\begin{itemize}
\item[(1)] The boundary of $\calB/\Gamma$ in $(\calB/\Gamma)^*$ consists of two points (called cusps), corresponding to the degenerations to cubic threefolds with $A_5$ and $D_4$ singularities respectively.
\item[(2)] $H_{c}=\calH_c/\Gamma$ and $H_{\Delta}=\calH_\Delta/\Gamma$, where $\calH_c$ and $\calH_\Delta$ are arithmetic arrangements of hyperplanes. The hyperplanes $\calH_\Delta$ and $\calH_c$ are reflective hyperplanes of $\Gamma$ corresponding to complex reflections of order $3$ and $6$ respectively (\cite[Lemma 2.3, Lemma 4.3]{act}).
\item[(3)] No two hyperplanes from $\calH_{c}$ meet inside $\calB$. The intersection of a hyperplane from $\calH_{c}$ and $\calH_\Delta$ is orthogonal (\cite[Thm 8.2]{act}).
\item[(4)] If a number of hyperplanes from $\calH_{\Delta}$ meet, they meet according to the Dynkin graph $A_{k_1} +\dots + A_{k_l}$ with $k_j\le 4$. The resulting locus in $\calB/\Gamma$ corresponds to either cubic threefolds with the combination $A_{k_1} +\dots + A_{k_l}$ of singularities, or to coincidences of the $12$ points on the rational normal curve of type $[(k_1+1), \dots, (k_l+1)]$ . The two situations are distinguished arithmetically by the condition that a hyperplane from $\calH_c$ either passes or not through the intersection (e.g. \cite[Lemma 3.3]{act}). 
\item[(5)] A combination of allowable singularities for cubic threefolds can be deformed independently. This is due to the fact that the singularities of cubic threefolds correspond to the way that the hyperplanes from $\calH_{\Delta}$ intersect  (see \cite[Lemma 2.4, Theorem 3.1, 3.2]{act}).
\end{itemize}

\section{Cubic Threefolds and Curves}\label{sectioncubics}
\subsection{Geometry of cubic threefolds with isolated singularities} \label{sec23} 
Let $X$ be a cubic threefold and $p$ be a singular point of $X$. The  linear projection $\pi_p$ with center $p$ gives a birational isomorphism between $X$ and $\bP^3$, which is resolved  by the blow-up of the point $p$. We obtain the  commutative diagram:
\begin{equation*}
\xymatrix{
&Q_p \ar@{->}[ld]\ar@{^{(}->}[r]&\widetilde{X}\ar@{<-^{)}}[r] \ar@{->}[ld]_{f} \ar@{->}[rd]^{g}&E_p\ar@{->}[rd] \\\
p\ar@{^{(}->}[r]&X \ar@{-->}[rr]^{\pi_p}&&\bP^3 \ar@{<-^{)}}[r] &C_p 
}
\end{equation*}
where  $Q_p$ is the projectivized tangent cone at $p$,  $C_p$ is the scheme parametrizing the lines on $X$ passing through $p$, and $E_p$  is the exceptional divisor of $g$. 

\begin{pro}\label{propblowup}
Let $X$ be a cubic threefold with isolated singularities and  $p\in X$ a singular point of corank at most $2$. With notations as above, we have:
\begin{itemize}
\item[i)] $C_p$ is a proper $(2,3)$ complete intersection in $\bP^3$; 
\item[ii)] $C_p$ is reduced (but possibly reducible);
\item[iii)] $C_p$ has only hypersurface singularities;
\item[iv)] the singularities of $C_p$ are in one-to-one correspondence,  including the type,  with the singularities of $\widetilde{X}$.
\end{itemize}
Furthermore, $\widetilde{X}$ is the blow-up of $\bP^3$ along the reduced scheme $C_p$.
\end{pro}
\begin{proof} The proposition is well known. For instance, for the statement about the analytic type of the singularities of $C_p$  we refer to Wall \cite[pg. 7]{wallsextic}. 
\end{proof}

The previous proposition makes possible a complete analysis of the possible singularities of a cubic threefold. In particular, we note the following consequence. 
\begin{cor}\label{lemtype}
Let $X$ be a cubic threefold and $p$ an isolated singular point. Then,
\begin{itemize}
\item[i)] the singularity at $p$ is $A_1$ if and only if $Q_p$ is a smooth quadric;
\item[ii)] the singularity at $p$ is $A_2$ if and only if $Q_p$ is a quadric cone and $C_p$ does not passes through the vertex of $Q_p$;
\item[iii)] the singularity at $p$ is $A_k$ for $k\ge 3$ if and only if $Q_p$ is a quadric cone, $C_p$ passes through the vertex $v$  of $Q_p$, and  the singularity of $C_p$ at $v$ is $A_{k-2}$;
\item[iv)] the singularity at $p$ is $D_4$ if and only if $Q_p$ is the union of two distinct planes and $C_p$ meets the singular line of $Q_p$ in three distinct points.
\end{itemize}
\end{cor}
\begin{proof}
The statement about $Q_p$ is a statement about the corank of the singularity of $X$ at $p$ (see \cite[Ch. 16]{agv1}). The blow-up of an $A_k$ (resp. $D_4$) singularity gives an $A_{k-2}$ singularity (resp. $3$ ordinary nodes). The result follows then from \ref{propblowup}.
\end{proof}

\begin{cor}\label{cortype}
The projection from a singular point of a cubic threefold gives a natural correspondence between:
\begin{itemize}
\item[i)] Smooth non-hyperelliptic genus $4$ curves and cubic threefolds with a unique singularity of type $A_1$ or $A_2$. The singularity is $A_2$ if and only if the corresponding curve has a unique $g_3^1$.
\item[ii)] Smooth hyperelliptic genus $3$ curves with $2$ marked points (not necessarily distinct)  and cubic threefolds with a unique singularity of type $A_3$ or $A_4$. The singularity is $A_4$ if and only if the two points coincide. 
\end{itemize}
\end{cor}
\begin{proof}
The situation for i) is clear.  For ii), we note  that by \ref{lemtype} the resulting curve $C_p$ lies on a quadric cone and has a singularity of type $A_1$ or $A_2$ at the vertex. The normalization $N(C_p)$ is of genus $3$ and is hyperelliptic due to the ruling of the cone.  The two marked points are the preimages of the node (counted with multiplicity two for the cusp). 

Conversely, given $N$ a hyperelliptic curve of genus $3$ with two marked point $p_1$ and $p_2$ we can map $N$ to $\bP^3$ so that the image is a $(2,3)$ curve, with the points $p_1$ and $p_2$ at the vertex of a singular quadric, as in the Corollary.  We will sketch this in the $A_4$ case; the $A_3$ case is similar.  So suppose that $p_1=p_2=p$ is a general point on a general hyperelliptic curve $N$ of genus three, and consider the map to $\mathbb P^3$ given by the linear system
$|2g^1_2+2p|$.  The image $C$ will be cuspidal, and a general projection will exhibit the curve as a plane  sextic with six nodes and one cusp.  Blowing up the six nodes shows that $C$ lies on a cubic surface.  Similarly, one can find a projection of $C$ which exhibits it as a plane curve with a singularity formed by an $A_2$ and $A_3$ singularity with transverse tangent cones.  An appropriate blow-up exhibits $N$ as a curve on $\mathbb F_2$, and blowing down the $(-2)$-curve exhibits $C$ as a  curve on a rank three quadric in $\mathbb P^3$ with a cusp at the vertex of the cone.
\end{proof}

\subsection{Cubic threefolds and plane quintics}\label{secquintic}
Let $X$ be a cubic threefold with isolated singularities, and $\ell \subset X$  a fixed line. The linear projection with center $\ell $ expresses $X$ as a conic bundle over $\bP^2$. Specifically, the blow-up  $\bP_\ell^4$ of the ambient projective space along $\ell $ gives a commutative diagram:
\begin{equation}\label{diagconic}
\xymatrix{
X_\ell\ar@{->}[r]\ar@{->}[rd]_{\pi}&\bP_\ell^4\ar@{->}[d]^{\tau}\\ 
&\bP^2
}
\end{equation}
where $X_\ell$ is the strict transform of $X$, and $\tau$ and $\pi$ the linear projections with center $\ell $. The generic fiber of $\pi$ is a smooth conic; the degenerate fibers of $\pi$ are parametrized by a plane quintic $D$, called the discriminant curve. For our purposes we make the following standard genericity assumption on the line $\ell $ (e.g. \cite[\S2]{btheta}).

\begin{dfn}\label{deftype1}
Let $X$ be a cubic threefold, and $\ell \subset X$ a line. We say that $\ell $ is a {\it non-special line}  iff for every line $f\subset X$ meeting $\ell$, the plane spanned by $\ell $ and $f$ cuts on $X$ three distinct lines.
\end{dfn}

\begin{rem}\label{remtype1} 
In the language of Clemens--Griffiths \cite{cg},  $\ell $ is non-special iff:
\begin{itemize}
\item[1)] $\ell $ is not a line of second type (\cite[Def. 6.6]{cg}, \cite[1.13, 1.14]{murre1}); 
\item[2)] $\ell $ is not the residual of a line of second type.
\end{itemize}
We note that if either of the  following degenerate situations hold: 
\begin{itemize}
\item[i)] $\ell $ passes through a singular point
\item[ii)] there exists a plane $P$ such that $\ell \subset P\subset X$
\end{itemize} 
then $\ell $ is considered to be a line of second type.
\end{rem}

We recall that a choice of non-special line $\ell $ on a cubic $X$ determines in addition to the discriminant curve $D$ a double cover $\widetilde{D}\to D$. Namely, $\widetilde{D}$ parametrizes the two residual lines to $\ell $ in a degenerate plane section of $X$ through $\ell $. More intrinsically, $\widetilde{D}$ is obtained from the Stein factorization of $\pi_{\mid \pi^{-1}(D)}:\pi^{-1}(D)\to D$. We then have the following result (this is well known in the smooth case).

\begin{pro}\label{lemcubdisc}
Let $X$ be a cubic threefold and $\ell $ a non-special line on $X$.  Let $D$ and $\widetilde{D}$ be as above. Then the following hold:
\begin{itemize}
\item[i)] there exists a natural $1$-to-$1$ correspondence between the singularities of $D$ and those of $X$, including the analytic type;
\item[ii)] the double cover $\widetilde{D}\to D$ is \'etale.
\end{itemize} 
\end{pro}
\begin{proof}
Choosing coordinates $(x_0:\dots:x_4)$ on $\bP^4$ such that the line $\ell $ is given by $(x_2=x_3=x_4=0)$, we can write the equation of $X$ as 
\begin{equation}\label{eqf}
F(x_0,\dots,x_4)=\ell_1 x_0^2+2 \ell_2 x_0 x_1+\ell_3 x_1^2+ 2 q_1 x_0+2 q_2 x_1 +f 
\end{equation}
where $\ell _i$, $q_j$, and $f$ are homogeneous polynomials in $x_2,x_3,x_4$ of degree $1$, $2$, and $3$ respectively.  The discriminant quintic $D$ is given by the determinant of the matrix
\begin{equation}\label{eqa}
A=\left(\begin{array}{ccc}
\ell_1&\ell_2&q_1\\
\ell_2&\ell_3&q_2\\
q_1&q_2&f
\end{array}\right),
\end{equation}
where $x_2,x_3,x_4$ are considered also coordinates on $\bP^2$. In the plane $\bP^2$ there is an additional discriminant curve, a conic $Q$. Namely, for the exceptional divisor $E$ of the blow-up $\widetilde{X}\to X$ we have a double cover map $\pi_{\mid E}:E \to \bP^2$ with branch locus the conic $Q$ given by the determinant of the matrix 
\begin{equation}\label{eqb}
B=\left(\begin{array}{ccc}
\ell_1&\ell_2\\
\ell_2&\ell_3
\end{array}\right).
\end{equation}
The quintic $D$ and the conic $Q$ are everywhere tangent. In the case that $X$ is smooth and $\ell $ is generic, it follows that $D$ is a smooth plane quintic and that $Q$ cuts on $D$ a theta characteristics $\kappa$, which in turn determines the \'etale cover $\widetilde{D}\to D$. Conversely, given $D$ and $\kappa$ one recovers the cubic threefold $X$ (e.g. Beauville \cite[\S4.1]{bdet}, cf. Proposition \ref{lemcubb}). 

In the singular case, we define  the curves $D$ and $Q$ to be given by the determinants (\ref{eqa}) and (\ref{eqb}) respectively. As long as $\ell $ does not pass through a singular point of the cubic $X$, the geometric interpretation of $D$ and $Q$ as discriminant curves still holds. Also, since the computations of Clemens--Griffiths \cite[\S6]{cg} are local around $\ell \subset X$, many of the key results are still valid under mild genericity assumptions on $\ell $ (e.g. \cite[6.7, 6.19]{cg}). By a slight generalization we obtain the following result characterizing the assumptions on $\ell $ in terms of the  geometry of the discriminant curves $D$ and $Q$.

\smallskip

\noindent{\bf Claim 1.} {\it With notation as above.  The non-degeneracy assumptions 1) and 2) of remark \ref{remtype1} on the line $\ell $ are equivalent to 
\begin{itemize}
\item[1')] $Q$ is a smooth conic (in particular $q\neq 0$);
\item[2')] $Q$ does not pass through a singular point of $D$;
\end{itemize}
respectively.  In particular, if $\ell $ is non-special then $Q$ and $D$ are non-vanishing.\qed}

\smallskip

It is easy to see that a singular point $p$ on $X$ always projects to a singular point $\pi(p)$ of the discriminant curve $D$. Furthermore, under the genericity assumption on $\ell $  no two singular points of $X$ map to the same point of $D$. Namely, assuming  that $p_1$ and $p_2$ are two singular points on $D$ such that $\pi(p_1)=\pi(p_2)$ (equivalently the lines $\ell $ and $\langle p_1,p_2\rangle$ are coplanar), one checks that either $\ell $ passes through one of singular points $p_1$ or $p_2$,  or $\ell $ lies in a plane which is contained in $X$. In conclusion, we only have to check that  every singular point of $D$ is obtained as a projection of a singularity of $X$. Additionally, we want that the analytic type is preserved.

\smallskip

\noindent{\bf Claim 2.}
{\it Assume that $\ell $ is a non-special line. If $\bar{p}\in D$ is a singular point of the discriminant curve, then there exists  an open subset $\bar{p}\in U\subset \bP^2$ such that $X_U=\pi^{-1}(U)$ is defined in $U\times \bP^2$ by the equation: $x_0x_1+c \cdot t^2$, where $c\in H^0(U,\calO_U)$ is a local equation of $D$ at $\bar{p}$, and $x_0,x_1,t$ are homogeneous coordinates on $\bP^2$.}
\begin{proof} 
Let  $q=\det(B)$ the equation of the conic $Q$. By the first claim, $q$ does not vanish at $\bar{p}$. It follows that on a suitable neighborhood $U$ of $\bar{p}$ the matrix 
\begin{equation}\label{eqm}
M=\left(\begin{array}{ccc}
\frac{\ell_2 + \sqrt{q}}{2}& -\frac{\ell_1}{2}& 0\\ 
\frac{-\ell_2 + \sqrt{q}}{l1}\cdot\frac{1}{q}& \frac{1}{q}& 0\\
 \frac{\ell_3 q_1 - \ell_2 q_2}{q}& \frac{-\ell_2 q_1 + \ell_1 q_2}{q}&1
\end{array}\right)
\end{equation}
(with $\det(M)=\frac{1}{\sqrt{q}}$) is well defined with holomorphic entries. The change of variables defined by $M$ diagonalizes the matrix $A$, bringing the equation of $X_U$ to the specified form.
\end{proof}

The section $t=0$ is a local equation of the exceptional divisor $E\subset \widetilde{X}$. Thus, since $X$ is smooth along $\ell $, it follows that locally the equations of $X$ and $D$ are $x_0x_1+c(x,y)$ and $c(x,y)$ respectively. This establishes an obvious  correspondence between the singularities of $D$ and those of $X$. For the \'etale statement it is enough to note that near a singular point of $C$, the branches of $\widetilde{C}\to C$ are specified by the choice of branch for $\sqrt{q}$ and that the local equation for $\widetilde{C}$ is still $c$ (with notation as in  claim 2). \end{proof}

The case of degenerations to cubic threefolds with a unique singularity will play an important role for us. We note the following simple consequence. 

\begin{cor}\label{lemonesing}
Let $X$ be a cubic threefold with exactly one singular point, which is an allowable singularity, and $\ell $  a non-special line on $X$. Then the associated discriminant curve $D$ is irreducible.
\end{cor}
\begin{proof} By proposition \ref{lemcubdisc},  $D$ has a unique singularity, which is of type $A_k$ with $k\le 5$ or $D_4$.  On the other hand, if $D$ is reducible with a unique singular $p$, by B\'ezout, it follows that the singularity at $p$ is at least $A_7$. 
\end{proof}

We will want the following result which is due to Beauville.

\begin{pro}[Beauville]\label{lemcubb}
Let $\mathscr D\to \Delta$ be a family of smooth plane quintics degenerating to a plane quintic $\mathscr D_0=D$ which is reduced, and either nodal or irreducible.
Let $\nu:N\to D$ be the normalization.
Let $\mathscr M$ be a torsion free sheaf on $\mathscr D$ such that  $\mathscr M_t$ is a theta characteristic on $\mathscr D_t$ for all $t\in \Delta^\circ$  and $\mathscr M_0=\nu_*M$ for a theta characteristic $M$ on $N$.  
If  $h^0(\mathscr M_t)=1$ for all $t\in \Delta$, then there exists a family $\mathscr X\to \Delta$ of cubic threefolds, and a family of lines $\mathscr L\to \Delta$ such that for all $t\in \Delta$, we have $\mathscr L_t\subset \mathscr X_t$, and $(\mathscr D_t,\mathscr M_t)$ is the discriminant induced from projection from $\mathscr L_t$.
\end{pro}
\begin{proof}  
This statement is essentially contained in Beauville \cite[Prop. 4.2]{bdet} and \cite[Rem. 6.27]{bint}. The proof of the nodal case is written out by Friedman and the first author \cite[Thm. 4.1, Prop. 4.2]{cmf}. The proof can be easily adapted to the case of the irreducible plane quintic with singularities which are worse than nodes. The main point is that due to the fact that the normalization map $\nu$ is finite in the irreducible case, we have 
\begin{eqnarray*}
\mathscr {H}om_D(\nu_*M,\omega_D)&=&\nu_*(\mathscr Hom_N(M,\nu^!\omega_D))\\
&=&\nu_*(M^\vee\otimes \omega_N)=\nu_*M.
\end{eqnarray*}
(e.g.  \cite[Ex. III.6.10, III.7.2]{h}).  The proof then proceeds exactly as in \cite{cmf}.
\end{proof}

We close by noting that on a cubic threefold with allowable singularities there exist non-special lines. These lines can be moved in families. Together with the results stated above, it follows that the study of degenerations of cubic threefolds can be reduced to the study of degenerations of plane quintics. 

\begin{lem}\label{lemcubft}
Suppose that $X$ is a cubic threefold with only allowable singularities.  Then the Fano variety of lines $F(X)$ is a surface and the non-special lines form a (Zariski) open subset $U\subset F(X)$. 
\end{lem}
\begin{proof}
Under the given assumptions,  the Fano surface $F(X)$ is a reduced pure $2$-dimensional scheme (cf. \cite[Thm. 1.3]{ak}).
Let $D_1$ and $D_2$ be the sets of lines of second type and residuals of lines of second type respectively (see remark \ref{deftype1}). It is known that $D_1$ and $D_2$ are Zariski closed subsets.  The computations of Murre \cite[\S1]{murre1}, that show that $D_1$ and $D_2$ are divisors, are valid for the lines on $X$ that do not pass through singular points and which do not lie in a plane contained in $X$. It is easy to see that there exist components of $F(X)$ that do not correspond to planes contained in $X$. Also, it is known that the lines passing through singularities of $X$  form a closed $1$-dimensional subscheme of $F(X)$ (cf. \cite{ak}). The lemma follows.
\end{proof}

\begin{cor}\label{lemcubdef}
Let $\mathscr X \rightarrow \Delta$ be a family of cubic threefolds such that $X_t$ is smooth for all $t\in \Delta^\circ$, and $X_0$ has only allowable singularities.  If $\ell_0$ is a non-special line in $X_0$, then there is a family of lines $\mathscr L\rightarrow \Delta$ such that $L_t$ is a non-special line in $X_t$ for all $t\in \Delta$, and $L_0=\ell_0$.
\end{cor}
\begin{proof}
The associated family of Fano surfaces $\calF\to \Delta$ is flat (see  \cite{ak}). The special lines form a divisor $\calD\to \Delta$ which decomposes in a flat part plus a part supported on the central fiber. It is clearly possible to choose a local section of $\calF\to \Delta$ passing through $\ell_0$ and missing the divisor $\calD$. After possibly shrinking $\Delta$, we obtain a family of lines satisfying the  specified conditions.
\end{proof}

\subsection{Generalized Lefschetz pencils}

We will be interested in particular families of cubic threefolds $\mathscr X\to B$ for which we can compute certain numerical invariants.  In what follows,  $B$ will be a smooth curve containing a base point $0$. 

\begin{dfn} Suppose that 
$\pi:\mathscr X\to B$ is a family of projective hypersurfaces of degree $d$ in $\mathbb P^n$, over a smooth base $B$ of dimension one.
If $B$ meets the discriminant locus at smooth points, we say that the family is \emph{Lefschetz}.
If $B^\circ:=B-\{0\}$ meets the discriminant  at smooth points, and $B$ meets the discriminant exactly once at a general point of the locus of $A_k$ singular hypersurfaces, we 
say that the family is \emph{generalized Lefschetz} of type $A_k$.
We will say such a family is \emph{transverse} if in addition, $B$ meets the 
discriminant locus transversally at $0\in B$.
\end{dfn}

By transverse intersection we mean that the tangent space to the curve $B$ is not contained in the tangent cone to the discriminant divisor.
Essentially, a Lefschetz family is one for which all the fibers are either smooth, or have  a unique singular point which is a node, while a generalized Lefschetz family allows for the central fiber to have a unique $A_k$ singularity.  

Constructing such families is not hard:  let $X$ and $Y$ be  hypersurfaces of degree $d$ in $\mathbb P^n$, with $X$ smooth and general, and $Y$ general among those with an $A_k$ singularity.  If $F_X$ and $F_Y$ are homogeneous polynomials defining $X$ and $Y$, then $tF_X+sF_Y$ will define a transverse generalized Lefschetz pencil of hypersurfaces of type $A_k$. For the case of cubic threefolds, we will want to be able to project from a family of non-special lines in order to obtain a pencil of plane quintics.  One convenient way to construct a family of lines on our cubics is to require that the cubics $X$ and $Y$ contain a common line $\ell$.  We would then like to know if this can be done in such a way that the pencil of cubics is transverse, the line $\ell$ is non-special for all singular fibers, and the family of plane quintics is also transverse.

\begin{lem}\label{lema2cub}
Let $\mathscr X\to B$ be a general generalized Lefschetz pencil of cubic threefolds of type $A_2$, containing a fixed line $\ell\subset \mathbb P^4$. 
\begin{enumerate}
\item $B$ meets the discriminant locus at $X_0$ with multiplicity two.  In other words, such a pencil is transverse.
\item For each $b\in B$ such that $X_b$ is singular, the line $\ell$ is non-special for $X_b$.
\end{enumerate}
\end{lem}
\begin{proof} (1) This follows from the standard results on the discriminant (Milnor, Brieskorn-Pham, Teissier \cite{teissier}, see also Smith-Varley \cite[Corollary 5.12]{svdisc}) once one shows that for an $A_2$ cubic $X_0$ containing a fixed line $\ell$, there is 
an open neighborhood $V$ of $X_0$ in the $\mathbb P^{30}$ of cubics containing the line $\ell$ which is versal.  Since every cubic threefold is  isomorphic to one containing the line $\ell$, this essentially  comes down to showing that the differential of the map $V\to \overline{\mathcal M}$ is surjective at $X_0$.  Since $T_{X_0}V\subset T_{X_0}\mathbb P^{34}$, a dimension count using the fact that an open subset of $\mathbb P^{34}$ is versal gives the result.

(2)  Using the incidence correspondence of the universal Fano variety of lines for cubic hypersurfaces in $\mathbb P^4$, together with the fact that the special  lines form a one dimensional subscheme of the Fano variety of a fixed smooth cubic (see Lemma \ref{lemcubft}), it is easy to check that the locus of cubics containing $\ell$ for which $\ell$ is a special line, has codimension one.  The intersection of this locus with the discriminant gives a codimension two locus in the space of cubics containing a line.  We can take our pencil to miss this locus. 
\end{proof}

From the description of the plane quintic as a determinant of a three by three matrix, it is not hard to see that
projection from the line $\ell$ determines a degree three pencil of plane quintics $\mathscr D\to B$.  
This family will be generalized Lefchetz of type $A_2$ by virtue of the lemma above, together with Propositon \ref{lemcubdisc}, and Beauville \cite[Proposition 1.2 (iii)]{bint} stating that for any line on a smooth cubic threefold, projection from that line gives a discriminant curve with at worst nodes.

\begin{lem}\label{lema2qui}
Let $\mathcal D\to B$ be the generalized Lefschetz pencil of plane  quintics of type $A_2$ obtained via projection from $\ell$.  Then $B$ meets the discriminant locus at $C_0$ with multiplicity two.  In other words, the family of plane quintics is transverse.
\end{lem}
\begin{proof}
There is a birational map between the moduli space of pairs $(X,\ell)$ consisting of a cubic $X$ together with a line $\ell \subset X$, and the moduli of odd double covers of plane quintics.  A local computation then gives the result.
\end{proof}

\section{Intermediate Jacobians of cubic threefolds}\label{secjac} 

\subsection{Intermediate Jacobians and one parameter degenerations} From a cohomological point of view a cubic threefold behaves like a curve. Specifically, all the Hodge theoretic information of a cubic threefold $X$ is encoded in a principally polarized abelian variety, the intermediate Jacobian $J X:=H^{2,1}(X)^{\vee}/H_3(X,\bZ)$. Moreover, by the global Torelli theorem of Clemens and Griffiths \cite{cg} this information suffices to recover the cubic. To be  precise, the period map $\calP:\calM\to \calA_5$ which assigns to a cubic threefold its intermediate Jacobian is an isomorphism onto its image. We denote by $I$ the image of the period map and by $\bar{I}$ the closure of $I$ in the Satake compactification $\calA_5^*$.

Our aim is to understand the locus $\bar{I}$ in $\calA_5^*$. This is essentially equivalent to identifying the limit intermediate Jacobian for $1$-parameter degenerations  $\calX\to \Delta$ of cubic threefolds. As mentioned above, by using the  GIT compactification of the moduli space of cubic threefolds, we can restrict ourselves to a small list of  possibilities for the central fiber $X_0$. The case when $X_0$ is the secant variety of the rational normal curve in $\mathbb P^4$ was analyzed by Collino \cite[Theorem 0.1]{col}. In this case, the limit intermediate Jacobian $\lim_{t\to 0}JX_t$ is a hyperelliptic Jacobian.  Moreover, all hyperelliptic Jacobians arise as a limit in this way (see also \cite{casa} for a proof of this via Prym varieties and degenerations of plane quintics).  Therefore, we can restrict our attention to studying families $\mathscr X\to \Delta$ of cubic threefolds degenerating to cubic hypersurfaces with allowable singularities. 

To do this we recall the result of Mumford \cite{mprym} (cf. Beauville \cite{bint}) which says that the intermediate Jacobian $JX$ of a smooth cubic threefold $X$ is isomorphic as a ppav to the Prym variety $P(\widetilde{D},D)$ associated to the double cover of the plane quintic $D$ obtained by projecting $X$ from a non-special line. Combining this with the results of \S\ref{secquintic}, we see that we can analyze the limit intermediate Jacobian  by using the theory of degenerations of Prym varieties (esp. Beauville \cite{b}). The last geometric result that we need relates the Jacobian of the $(2,3)$ curve discussed in \S\ref{sec23} with Prym variety associated to the double cover of the quintic. The following is a slight generalization  of Collino-Murre \cite[Thm. 4.22]{cm}.

\begin{teo}[Collino-Murre]\label{teojac}
Let $X_0$ be a cubic threefold with a unique singular point $p$, and suppose that $p$ is of type $A_1$ or $A_2$.  Let $C\subset \mathbb P^3$ be the associated curve of type $(2,3)$.  For a general line $\ell_0\subset X_0$ the associated discriminant curve
$\widetilde{D}\to D$ is irreducible.  Denoting by $\widetilde{N}\to N$ the normalization, 
$$
JC \cong P(\widetilde{N},N)
$$
as ppavs.
     \end{teo}
\begin{proof}
The proof in Collino-Murre is for a singularity of type $A_1$.  The $A_2$ case can be deduced from this using a degeneration argument. 
Indeed, let $X_0$ be a cubic threefold with a unique allowable singularity $p$.  Let $C\subset \mathbb P^3$ be the associated curve of type $(2,3)$, and let $\hat{C}$ be its normalization.  For a general line $\ell_0\subset X_0$ the associated discriminant curve
$\widetilde{D}\to D$ is irreducible.  Denote by $\widetilde{N}\to N$ the normalization.  The Fano variety $F_{X_0}$ of lines on $X_0$ is birational to $\hat{C}^{(2)}$, the symmetric product, and since $\widetilde{D}\subset F_{X_0}$, it follows from the universal property of the Albanese that there is an induced map
$a:J\widetilde{N}\to J\hat{C}$.  This gives rise to a diagram
\begin{equation}\label{eqncolmur}
\begin{CD}
J\widetilde{N}@>a>> J\hat{C}\\
@V\textnormal{Nm}VV\\
JN\\
\end{CD}
\end{equation}
where $\textnormal{Nm}$ is the norm map induced from the map $\widetilde{N}\to N$.  In the case that the cubic has an $A_1$ singularity, Collino and Murre \cite{cm} study this situation to show that $P(\widetilde{N},N)\cong J\hat{C}=JC$.

Now let $\mathscr X\to \Delta$ be a family of cubic threefolds with a unique $A_1$ singularity, degenerating to a cubic threefold with a unique $A_2$ singularity.  In this case, the family of cubics gives rise to a family curves, and of diagrams as in (\ref{eqncolmur}), which implies the result.
\end{proof}

\begin{rem}\label{colmurconj} 
It seems likely that Theorem \ref{teojac} holds in more generality.  It may be possible through a careful analysis as in \cite{cm} to prove the theorem under the weaker hypothesis  that $X_0$ is a cubic threefold with a unique singular point $p$, which is of type $A_k$ $(1\le k\le 5)$ .  In this case, let $C\subset \mathbb P^3$ be the associated curve of type $(2,3)$, and let $\hat{C}$ be its normalization.  For a general line $\ell_0\subset X_0$ the associated discriminant curve
$\widetilde{D}\to D$ is irreducible.  Denoting by $\widetilde{N}\to N$ the normalization, the statement would be that 
$J\hat{C} \cong P(\widetilde{N},N)$
as ppavs.
Support for this is given by Proposition \ref{proirr} and Corollary \ref{lemtype}, from which one can check that 
$J\hat{C}$ and $P(\widetilde{N},N)$ are of the same basic type.  For instance, in the $A_4$ case, both $J\hat{C}$ and 
$P(\widetilde{N},N)$ are products of hyperelliptic Jacobians of dimensions two and three.    
\end{rem}

\subsection{Preliminaries on the boundary locus}

We start by noting the important fact that a degeneration of an intermediate Jacobian is a hyperelliptic Jacobian, or the product of Jacobians.  This follows from work of Debarre \cite{deb}, Matsusaka  \cite{mat}, Ran \cite{ran} and the first author \cite{casa}.

\begin{pro}\label{propre}
$\partial I \subset \bar{J}_5\subset \mathcal A_5^*$.
     \end{pro}
\begin{proof}
Suppose $(A,\Theta)\in \partial I$.  First consider the case that $(A,\Theta)\in \mathcal A_5$ is indecomposable.  \cite[Theorem 2]{casa}  then implies that $(A,\Theta)$ is isomorphic to the Jacobian of a hyperelliptic curve.

Now consider the case that $(A,\Theta)\in \mathcal A_5$ and there are ppavs $(A_1,\Theta_1)$ and $(A_2,\Theta_2)$ such that 
$(A,\Theta)\cong (A_1,\Theta_1)\times (A_2,\Theta_2)$.
A ppav of dimension less than four is a degeneration of a Jacobian; thus by symmetry, we may assume $\dim A_1=4$.
Let $\mathscr A \to \Delta$ be a family of intermediate Jacobians of cubic threefolds degenerating to $\mathscr A_0=(A,\Theta)$.
Over $\Delta^\circ$ there is a family of Fano surfaces contained in these abelian varieties
$$
\begin{CD}
\mathscr S^\circ @>>> \mathscr A\\
@VVV @VVV\\
\Delta^\circ  @>>> \Delta.
\end{CD}
$$
Taking the closure in $\mathscr A$ gives a family of surfaces 
$\mathscr S \to \Delta$.
Set $S=\mathscr S_0$.  The class of $S$ in $A$ must be the same as that of the Fano surface in  an intermediate Jacobian; i.e.  $[S]=[\Theta]^3/3!$.  On the other hand we can set $C_0=S|_{A_1}$, and by virtue of the fact that  $\Theta|_{A_1}=\Theta_1$,  it follows that $[C_0]=[\Theta_1]^3/3!$.  By the Matsusaka-Ran  criterion \cite{mat} \cite{ran}, $(A_1,\Theta_1)$ is a Jacobian.  Thus $(A,\Theta)\in \bar{J}_5$.   

Finally, suppose that $(A,\Theta)\in \mathcal A_4^*\subset \mathcal A_5^*$.   If $\dim A\le 3$, then $(A,\Theta)$ is a product of Jacobians, and so we may assume that $\dim A=4$.  Proposition 7.1 of Debarre \cite{deb} implies that  $(A,\Theta)$ contains a subvariety with minimal class.   In dimension four, results of Matsusaka  \cite{mat} and Ran \cite{ran}  show that this implies $(A,\Theta)$ is a Jacobian.
\end{proof}

The following lemma  was pointed out to the first author by Samuel Grushevsky \cite{grup}, who gave a proof using theta functions and results from \cite{gru}.
We provide a different proof here using more geometric methods.

\begin{lem}\label{lemdec2}
If $(A,\Theta)\in \partial I \cap \mathcal A_5$ is the product of two indecomposable ppavs, then it is either in  $\bar{J}_5^h$ or $K$.
     \end{lem}
\begin{proof} 
If $(A,\Theta)$ is the product of indecomposable ppavs of dimension two and three, then the existence of  a triple point on the theta divisor implies by the Riemann singularity theorem that both ppavs are hyperelliptic Jacobians.
Now suppose that $(A,\Theta)$ is the product of an elliptic curve and an indecomposable ppav of dimension four.  Due to Proposition \ref{propre}, the ppav of dimension four is the Jacobian of a smooth curve $C$.
The triple point on the theta divisor of the intermediate Jacobian of a cubic threefold is a two torsion point of the torus, and thus $(A,\Theta)$ has such a triple point as well.  This is only possible if $C$ has a unique $g^1_3$, and thus
$(A,\Theta)\in K$.
\end{proof}

\section{Prym varieties and degenerations of intermediate Jacobians}\label{secpry}
\subsection{Limits of Prym varieties associated to discriminant curves of cubic threefolds}\label{seclimit}
We start with a triple $(X_0,\ell_0,\mathscr X \to \Delta)$ consisting of a cubic threefold $X_0$ with allowable singularities, $\ell_0\subset X_0$ a general line on $X_0$, and $\mathscr X \to \Delta$ a family of smooth cubic threefolds degenerating to $\mathscr X_0=X_0$.
 Lemma \ref{lemcubft} and Lemma \ref{lemcubdef} 
imply that there exists a family of lines $\mathscr L \to \Delta$ such that $\mathscr L_t\subset \mathscr X_t$ for all $t\in \Delta$, $\mathscr L_0=\ell_0$,  and  projection from this family of  lines induces a family of discriminant curves 
  $$(\pi:\widetilde{\mathscr D}\to \mathscr D)\to \Delta$$ consisting of odd connected \'etale double covers of smooth plane quintics degenerating to a connected double cover of a plane quintic $D_0$ with allowable singularities.

The family of double covers induces a map $\Delta^\circ\to \overline{\mathcal{R}}_g$.  Since $\overline{\mathcal{R}}_g$ is proper, there is a unique extension of the map to $\Delta$, and (possibly after an appropriate base change) 
let $$(\widetilde{\mathscr C}\to \mathscr C)\to \Delta$$
 be the 
corresponding family, which we will call an ``admissible reduction.''  
For the sake of computing an admissible reduction,  we remark that since the central fiber in a stable reduction of a family of curves is unique, it follows that 
$\widetilde{\mathscr C}\to \Delta$ 
and $\mathscr C\to \Delta$ are  stable reductions of the families $\widetilde{\mathscr D}\to \Delta$ and $\mathscr D\to \Delta$ respectively.  Recall also that the central fiber $C_0$ will consist of the normalization of $D_0$ at singularities worse than nodes, together with other irreducible curves called tails.

\begin{lem}\label{lemexc}
Given a triple 
$(X_0,\ell_0,\mathscr X\to \Delta)$ as above, the central fiber of the admissible double cover $\pi:\widetilde{C}_0\to C_0$ satisfies the following conditions:
\begin{enumerate}
\item No  point of $\widetilde{C}_0$ is fixed by the involution,

\item Each irreducible component of $\widetilde{C}_0$ arising as a tail in thes stable reduction is exchanged by the involution. 
\end{enumerate}
Moreover, $\lim_{t\to 0}JX_t=P(\widetilde{C}_0,C_0)$.
     \end{lem}
\begin{proof}
 Since the line $\ell_0$ is non-special, the fiber of $\widetilde{D}_0\to D_0$ over each point of $D_0$ consists of two points.  This proves (1).  Let $p$ be a singular point of $D_0$, and let $T_1$ and $T_2$ be the tails arising from the stable reduction of the points lying above $p$ on $\widetilde{D}_0$.  
Since $\widetilde{C}_0$ is an admissible double cover of $C_0$, it is easy to check that neither $T_1$ nor $T_2$ is empty, and that $T_1$ and $T_2$ are exchanged by the involution.
\end{proof}

\begin{pro}\label{proirr}
Let $X_0$ be a cubic threefold with a unique allowable singularity, and let $\ell_0  \subset X_0$ be a general line on $X_0$.  The discriminant curve $D_0$ is irreducible, and setting 
$\widetilde{N}\to N$ to be the normalization of the  covering of $D_0$, there is a curve $\Sigma$ of compact type such that 
$P(\widetilde{N},N)\cong J\Sigma$. 
Moreover, 
\begin{enumerate}
\item Suppose that   $\mathscr X\to \Delta$ is a family of smooth cubic threefolds degenerating to $X_0$.
\begin{enumerate}
\item  If $X_0$, and hence $D_0$, has an $A_{2k}$ $(k>0)$ singularity, then
$$P(\widetilde{C}_0,C_0)\cong P(\widetilde{N},N)\times T \cong J\Sigma\times T,$$
 where  $T$ is the hyperelliptic curve of genus $k$  arising as the tail in the  stable reduction of the associated family of plane quintics.
\item  If $X_0$, and hence $D_0$, has an $A_{2k+1}$ $(k>0)$, or $D_4$ singularity, then  $P(\widetilde{C}_0,C_0)$ is  a $\mathbb C^*$-extension.
\end{enumerate}
\item If $X_0$, and hence $D_0$, has an $A_{2k}$ $(k>0)$ singularity, then for any hyperelliptic curve $T$ of genus $k$ there is a family of cubic threefolds $\mathscr X\to \Delta$ degenerating to $X_0$ such that 
$$\lim_{t\to 0} J X_t\cong P(\widetilde{N},N)\times T \cong J\Sigma\times T.$$
\end{enumerate}
\end{pro}
\begin{proof}  
Take $\widetilde{N}\to N$ as in the statement of the proposition.  Lemmas \ref{lemcubdisc} and \ref{lemonesing} imply that $D_0$ is irreducible and has a unique singular point.  This singular point implies that $N$ is trigonal, and a result of Recillas \cite{rec} implies that $P(\widetilde{N},N)$ is a Jacobian.  We point out for later reference that $N$ is not hyperelliptic since $g(N)\ge 3$, and the $g^1_3$ induced by the singular point is base point free.

(1) (a).
We start by assuming we are given the family $\mathscr X\to \Delta$, and that the discriminant curve $D_0$ has an $A_{2k}$ singularity.
It follows from the work of Hassett \cite[Section 6]{has}  that the stable reduction of $D_0$ consists of the normalization $N$ joined at one point to the tail $T$, which is a hyperelliptic curve of genus $k$.  Because $\pi:(\widetilde{\mathscr D}\to \mathscr D)\to \Delta$ is flat, the cover $\widetilde{D}_0\to D_0$ is connected which in turn implies that the cover $\widetilde{N}\to N$ is connected, since the singularities are of type $A_{2k}$.   It follows from Beauville \cite{b} Theorem 5.4 that 
$P(\widetilde{C}_0,C_0)=P(\widetilde{N},N)\times T$.

(2) Conversely, suppose that we are given $X_0$, and the covering of the discriminant $\widetilde{D}_0\to D_0$.  
By Hassett's results, since $D_0$ has an $A_{2k}$ singularity, given a hyperelliptic curve
$T$ of genus $k$, there is a family $\mathscr D\to \Delta$ of plane quintics degenerating to $D_0$ with $T$ as the tail.
Given the family of smooth plane quintics degenerating to $D_0$, there is a family 
$(\pi:\widetilde{\mathscr D}\to \mathscr D)\to \Delta$
of odd double covers degenerating to $\widetilde{D}_0\to D_0$.  It is a result of Beauville's (see Proposition \ref{lemcubb}) that the family of double covers induces a family of smooth cubic threefolds $\mathscr X\to \Delta$ degenerating to $X_0$.  (2) then follows from the arguments used above to prove (1).

(1) (b).  For $D_4$ singularities, $C_0$ has two irreducible components, meeting in three points.  From 
Beauville \cite[Lemma 5.1]{b}  it follows that 
the Prym is a $\mathbb C^*$-extension. 
For $A_{2k+1}$ singularities,  due to \cite[Lemma 5.1]{b}, (2) is equivalent to the statement that the covering $\widetilde{N}\to N$ is connected.  We now prove each case.
For $A_1$ singularities, (2) is proven in Clemens-Griffiths \cite{cg} (cf. Collino-Murre \cite{cm}).  For $A_5$ singularities, Beauville \cite{b} implies that if  $P(\widetilde{C}_0,C_0)$  were not a $\mathbb C^*$-extension, then it would be isomorphic to $JN\times JT$.  This would contradict Lemma \ref{lemdec2} since $g(T)=2$, and $N$ is not hyperelliptic.  

For $A_3$ singularities, we use a degeneration argument.  Consider a family of cubic threefolds $\mathscr X\to \Delta$ such that the general fiber has an $A_3$ singularity, and the central fiber has an $A_4$ singularity.  This gives rise to a family of discriminant curves
$(\widetilde{\mathscr D}\to \mathscr D)\to \Delta$, which can be viewed as a family  of genus four curves together with the data of a linear system giving the map to the plane, and an \'etale cover of the curve.  In the proof of (1) (a) we showed that the covering corresponding the the central fiber is connected (the $A_4$ case).  The same is then true for the general fiber.
\end{proof}

\subsection{Divisorial boundary components}\label{secdivisors}
With the results of the previous section, we can show that $K$, $\bar{J}_4$, and $\bar{J}_5^h$ are contained in the boundary.  We begin by  showing that $K\subset \partial I$.

\begin{cor}\label{cordiv}
If $X_0$ is a cubic threefold with a unique singularity, which is of type $A_2$, and  $\mathscr X\to \Delta$ is a family of smooth cubic threefolds degenerating to $X_0$,  
then $\lim_{t\to 0}JX_t \cong JC\times E$ where
$C$  is a smooth genus four curve admitting a unique $g^1_3$, and  $E$ is an elliptic curve.

Moreover, given  a smooth genus four curve $C$ admitting a unique $g^1_3$, and  an elliptic curve $E$, there is  a cubic threefold $X_0$ with a unique singularity, which is of type $A_2$, and  a family $\mathscr X\to \Delta$  of smooth cubic threefolds degenerating to $X_0$, such that $\lim_{t\to 0}JX_t \cong JC\times E$.
\end{cor}

\begin{proof}
This follows from Proposition \ref{proirr}, Lemma \ref{lemtype} and Theorem \ref{teojac}.
\end{proof}

\begin{rem}\label{remclem}
Recall that results of Clemens and Griffiths \cite{cg} show that
$\bar{J}^4\subseteq \partial I \cap \mathcal A_4^*$.  Indeed for any smooth non-hyperelliptic curve $C$ of genus four, there is a cubic threefold $X_0$ with a unique singularity, which is of type $A_1$, and a family of smooth cubic threefolds 
$\mathscr X\to \Delta$ such that $\lim_{t\to 0}JX_t=JC$.
On the other hand, Proposition \ref{propre} (or more precisely Debarre \cite[Proposition 7.1]{deb})  implies that 
$\partial I \cap \mathcal A_4^* \subseteq \bar{J}^4$, so in fact $\partial I \cap \mathcal A_4^* = \bar{J}^4$.
\end{rem}

\begin{pro}\label{proinc}
$\bar{J}^h_5\cup \bar{J}^4 \cup K \subseteq \partial I$.
     \end{pro}
\begin{proof}
As pointed out earlier, Collino \cite[Thoerem 0.1]{col} implies that $\bar{J}^h_5\subset \partial I$, and 
the results of Clemens-Griffiths \cite{cg} (cf. Remark \ref{remclem}) imply that $\bar{J}_4\subset \partial I$.  Corollary \ref{cordiv} implies that
$K\subset \partial I$.
\end{proof}

\subsection{Proof of Theorem \ref{mainthm1}}\label{secproof}
To conclude the proof of Theorem \ref{mainthm1},  we must prove the converse of the inclusion given by Proposition \ref{proinc}.  The key observation is the following lemma, which is a consequence of the fact that $K\subset \partial I$.

\begin{lem}\label{cor3}
If $(A,\Theta)\in\mathcal A_5$ and $(A,\Theta)\cong \prod_{i=1}^k(A_i,\Theta_i)$ for some $k\ge 3$,
then $(A,\Theta)\in \partial I$.
     \end{lem}
\begin{proof}
In the notation of the corollary, if $k\ge 4$, then
$(A,\Theta)\in \bar{J}^h_5 \subset \partial I$.  If $k=3$, then the only case to consider is the case
$E_1\times E_2\times JC$ for $C$ a smooth non-hyperelliptic genus three curve, and $E_i$ $i=1,2$ elliptic curves. 
It is enough to show that given a general smooth curve $C$  of genus $3$, and an elliptic curve $E$, there is a family of genus four curves with a unique $g^1_3$ degenerating to a curve obtained from $C$ and $E$ by identifying one point of each curve.  Indeed this would show that $E_1\times E_2\times JC\in \partial K$.

Let $V_4=\{[\Sigma]\in \mathscr M_4: \Sigma \textnormal{ admits a unique } g^1_3\}$, and let $\overline{V}_4$ be the closure of $V_4$ in $\bar{\mathscr M}_4$.  The intersection $\overline{V}_4\cap \Delta_1$ is  equidimensional of dimension seven.  If $X=(C,p)\cup (E,q)/p\sim q\in \overline{V}_4 \cap \Delta_1$, then $X$ admits a unique limit $g^1_3$.  It is easy to check that for a general curve $C$, which can be viewed as a plane quartic, this condition is satisfied if and only if $p$ is a flex point, or lies on a bitangent to $C$ (i.e. the tangent line to $C$ at $p$ either meets $C$ to order $3$ at $p$, or is tangent to $C$ at another point).  A six dimensional irreducible component of the set of pointed curves $\{(C,p)\in \mathscr M_{3,1}: p \textnormal{ is a flex or lies on a bitangent}\}$ dominates $\mathscr M_3$.
\end{proof}

\begin{proof}[Proof of Theorem \ref{mainthm1}]
Suppose $(A,\Theta)\in \partial I$.  By virtue of Proposition \ref{proinc}, it suffices to show that $(A,\Theta)\in \bar{J}_4\cup \bar{J}_5^h\cup K$.  To begin, if $(A,\Theta)\in \mathcal A_4^*$, then due to the results of Clemens and Griffiths \cite{cg} referred to in Remark \ref{remclem}, $(A,\Theta)\in \bar{J}_4$.
So we will assume that $(A,\Theta)\in \mathcal{A}_5$.  If $(A,\Theta)$ is indecomposable,  \cite[Theorem 2]{casa}  implies that it is a hyperelliptic Jacobian.  If $(A,\Theta)$ is the product  of two indecomposable ppavs, then Lemma 
\ref{lemdec2} implies that $(A,\Theta)$ is in $\bar{J}^h_5$ or in $K$.  Finally, if $(A,\Theta)$ is the product of three or more ppavs,  Lemma \ref{cor3} implies it is in $K$.
\end{proof}

\begin{rem}\label{remtor}
One can also consider the closure of $I$ in the toroidal  compactifications of $\mathcal A_5$.  Following Friedman-Smith \cite{fs} or  Alexeev-Birkenhake-Hulek \cite{abh} one can give similar arguments using the degenerations of Prym varieties to understand the boundary of $I$.  It is well known
(cf. \cite{fs}, Debarre \cite[Proposition 7.1]{deb}, or Proposition \ref{propre}) that the limits will lie over the Jacobian locus  $\bar{J}_4\subset\mathcal A_4^*$, and some examples such as when the discriminant curve has a unique singular point should be accessible from the results of   \cite{fs} and \cite{abh} using an analysis similar to that of Proposition \ref{proirr}.
There are several difficulties which arise in the general case, however.  To begin, it is established in \cite{fs} that the Prym map $\mathcal  P:\mathcal  R_6\to \mathcal  A_5$ does not extend to a morphism $\overline{\mathcal  R}_6\to \overline{\mathcal  A}_5$, where $\overline{\mathcal{A}}_5$ is any of the toroidal compactifications; this makes more difficult the construction of an explicit and ``geometric'' resolution of the rational period map $\overline{\mathcal M}\dashrightarrow \overline{\mathcal A_5}$. 
In addition, when the discriminant curve has many irreducible components it may be hard to identify exactly which  compactifications arise in the limit.     See for example Gwena \cite{g} for some examples along these lines.  In our computation we avoid both of these difficulties:  the Prym map does extend to a morphism $\overline{\mathcal  R}_6\to \mathcal  A_{5}^*$ \cite[Proposition 1.8]{fs}, and the degeneration of a Jacobian is a product of Jacobians.  This means that in order to understand the boundary in $\mathcal  A_4^*$, we do not have to compute the exact limits that arise in every case. 
\end{rem}

\section{The birational geometry of the moduli space of cubic threefolds}\label{secbirational}
As mentioned in Section \ref{secgit} and the Introduction, the moduli space of cubic threefolds $\calM$ has three natural compactifications: $\overline{\calM}$ constructed in \cite{A,yokoyama} via geometric invariant theory, $(\calB/\Gamma)^*$ constructed in \cite{act,ls} via period maps for special cubic threefolds, and $\bar{I}$ given by the closure of the intermediate Jacobian locus.  All three spaces are birationally equivalent projective varieties, containing $\calM$ as a common open Zariski subset. In this section we discuss the precise nature of these birational equivalences, as well as their geometric meaning.  This builds on the description of the resolution of the rational map $\overline{\mathcal M} \dashrightarrow (\mathcal B/\Gamma)^*$ given in \cite{act,ls}.

\subsection{A resolution  of the period map}\label{secresolution} We start by noting that as a result of the Borel extension theorem \cite{borel}, in order to resolve the period map it suffices to produce a normal crossing compactification of the moduli space of cubic threefolds. A priori this might be difficult, but we have at our disposal the compactifications $\overline{\calM}$ and $(\calB/\Gamma)^*$ constructed in \cite{act} and \cite{ls}. These compactifications are not normal crossing, but  it is quite easy to make them so. We proceed as follows (N.B.  the process is summarized in Diagram \ref{resdiag}).

\subsubsection{Consider the blow-up  $\widehat{\calM}$ of $\overline{\calM}$  constructed in \cite{act} and \cite{ls}.} This is described in section \ref{secgit} above. The space $\widehat{\calM}$ dominates both  $\overline{\calM}$ and $(\calB/\Gamma)^*$ and can be viewed in two ways. On the one hand it can be viewed as an intermediate step in the Kirwan desingularization \cite{kirwan} of the GIT quotient $\overline{\calM}$ (see \cite[\S4]{act}).  Dually, $\widehat{\calM}$ is a small partial resolution of $(\calB/\Gamma)^*$  which makes the discriminant divisors $H_{\Delta}$ and $H_{c}$ $\bQ$-Cartier (see Looijenga \cite{l1}). 
\subsubsection{Consider a toroidal compactification $\overline{\calB/\Gamma}$ compatible to $\widehat{\calM}$.} By the construction of Looijenga \cite{l1}, the space $\widehat{\calM}$ is an intermediate blow-up between the toroidal compactification and the Baily-Borel compactification. Thus, $\overline{\calB/\Gamma}$ is obtained from $\widehat{\calM}$ by proceeding to a full blow-up of the boundary of the Baily--Borel compactification (see \cite[\S4.3]{l1}). We note that $\overline{\calB/\Gamma}$ dominates $\widehat{\calM}$ and is smooth (up to finite quotient singularities). There are two (non-intersecting) exceptional divisors corresponding to the two cusps of $(\calB/\Gamma)^*$, which we call $D_{A_5}$ and $D_{D_4}$ respectively. Also, dually, the space $\overline{\calB/\Gamma}$ can viewed as the full Kirwan desingularization of $\overline{\calM}$.
\subsubsection{Consider a minimal blow-up of $\overline{\calB/\Gamma}$ such that the pull-back of divisors $H_\Delta$  and $H_c$ are normal crossings.} Since $H_c$ and $H_\Delta$ are induced by arithmetic arrangements of hyperplanes $\calH_c$ and $\calH_\Delta$, there is a simple way to make them normal crossing divisors: blow-up the loci of self-intersection, starting with the highest codimension stratum (see \cite[\S2]{l1}).  We consider a slight variation of this construction as described by Looijenga  \cite[\S2.3]{l1}. Namely, we only blow-up the self-intersection loci of  the hyperplanes in $\calH_{\Delta}$ that are not normal crossing (e.g. we do not blow-up  the intersection of a hyperplane from $\calH_{\Delta}$ and one from  $\calH_{c}$, or the intersection of two hyperplanes from $\calH_{\Delta}$ corresponding to cubic threefolds with two ordinary nodes; these are already normal crossing). In particular, among the  exceptional divisors of $\widetilde{\calM}\to \overline{\calB/\Gamma}$ there are $3$ irreducible divisors  $D_{A_4}$, $D_{A_3}$ and $D_{A_2}$ which are obtained by blowing-up the loci in $\calB/\Gamma$ whose generic points correspond to cubics with a unique $A_4$, $A_3$, and $A_2$ singularity respectively (N.B. the  blow-up is in this order). 

\medskip

We conclude that the period map sending a cubic threefold to its intermediate Jacobian extends to a regular morphism from $\widetilde{\calM}$ to $\calA_5^*$ fitting in the following commutative diagram:

\begin{equation}\label{resdiag}
\xymatrix{
&&&\widetilde{\calM}\ar@{->}[ld]_{\epsilon_3}\ar@{->}[lddd]_(0.6){\widetilde{\calP}_c}\ar@{->}[rddd]^{\widetilde{\calP}}\ar@/_2.5pc/[lllddd]_{f}&\\
&&\overline{\calB/\Gamma} \ar@{->}[ld]_{\epsilon_2}&& \\
&\widehat{\calM} \ar@{->}[ld]_(0.4){\epsilon_1}\ar@{->}[rd]^{\widehat{\calP}_c}&& \\
\overline{\calM} \ar@{-->}[rr]^{\overline{\calP}_c}&&(\calB/\Gamma)^* \ar@{-->}[rr]&&\bar{I} &\\
&&\calM\ar@{^{(}->}[llu] \ar@{^{(}->}[u]^{\calP_c} \ar@{^{(}->}[rru]^{\calP} &&   &
}
\end{equation}

\begin{teo}\label{teoresolution}
The  blow-up $f:\widetilde{M}\to \overline{\calM}$ constructed above resolves the period map, i.e. the period map $\calP:\calM\to I$ extends to a morphism $\widetilde{\calP}: \widetilde{M}\to \bar{I}$.
\end{teo}
\begin{proof}
The question is local.  We can also ignore the finite quotient issues. The space $\overline{\calB/\Gamma}$ is essentially smooth by construction. The third step consists of blowing-up smooth subvarieties  in a smooth ambient space according to the procedure of  \cite{l1} (based on the structural results about the discriminant in \cite{act,ls}). It follows  that $\widetilde{\calM}$ is a normal crossing  compactification (up to finite quotient issues\footnote{i.e. locally $\widetilde{\calM}$ is the quotient of a smooth space by a finite group and the pull-back of the boundary divisor to this space is a normal crossing divisor}) of $\calM$. We conclude that, by the Borel extension Theorem \cite[Thm. A]{borel}, the period map extends to a morphism  $\widetilde{P}:\widetilde{M}\to \calA_5^*$. The image is clearly $\bar{I}$.
\end{proof}

\begin{rem}
The major difficulty in \cite{act,ls} encountered in resolving the period map $\overline{\calP_c}:\overline{\calM}\dashrightarrow (\mathcal  B/\Gamma)^*$ is to understand the structure of the discriminant hypersurface for cubic threefolds, and to understand the degenerations to the secant threefold. The compactification $(\calB/\Gamma)^*$ actually takes care of both these issues. This makes our situation much more manageable. The key results  that we use are those on the structure of the divisors $H_{c}$ and $H_{\Delta}$ discussed at the end of section \ref{secgit}.
\end{rem}

The space $\widetilde{\calM}$ is essentially smooth, thus one can make intersection theory computations of $\widetilde{\calM}$. Also, taking the point of view that $\widetilde{\calM}$ is obtained from a blow-up of a GIT quotient followed by blow-ups with  smooth centers,   one gets a good control on the cohomology of  $\widetilde{\calM}$ (e.g. see Kirwan et al. \cite{kirwan2}). For us, it is important to note that the Picard number of $\overline{\calM}$ is $1$ (see Mumford \cite[pg. 63]{mumford2} and note that the cubics are parametrized by a projective space) and thus that of $\widehat{\calM}$ is $2$.

\subsection{Geometric interpretation of the resolution}\label{geomint}
\begin{teo}\label{teodiv}
The divisors $H$, $D_{A_1}$ and $D_{A_2}$ on $\widetilde{\calM}$ are mapped onto the divisors $\bar{J}^h_5$, $\bar{J}_4$ and $K$ respectively in $\bar{I}$. The  boundary  divisors  $D_{A_3}$, $D_{A_4}$, $D_{A_5}$ and $D_{D_4}$ in $\widetilde{\calM}$ are contracted by the extended period map $\widetilde{\calM}$.
     \end{teo}

\begin{proof} The statement about $H$ is due to Collino \cite[Theorem 0.1]{col}, while the statement about $D_{A_1}$ is proven in Clemens-Griffiths \cite{cg} (cf. Remark \ref{remclem} and Collino-Murre \cite[Theorem 4.22]{cm}).  The statement about $D_{A_2}$ is proven in Corollary \ref{cordiv}.
  We show here that the other boundary divisors are contracted by the extended period map.  Indeed, let $\mathscr X\to \Delta$ be a one parameter family of smooth cubic threefolds degenerating to a cubic $X_0$ with a unique allowable singularity.  As before, this gives rise to a family $(\widetilde{\mathscr C}\to \mathscr C)\to \Delta$ of odd connected \'etale double covers of smooth plane quintics, degenerating to an admissible double cover.    Let $\widetilde{N}\to N$ be the normalization of the map $\widetilde{C}_0\to C_0$.  Since the covering of the normalization of the plane quintic associated to $X_0$ is connected and \'etale  and the covering of the tail is disconnected (Proposition \ref{proirr}), by Friedman-Smith \cite[Theorem 1.12]{fs}  there is an extension
$$
0\to (\mathbb C^*)^k\to P(\widetilde{C}_0,C_0)\to P(\widetilde{N},N)\to 0
$$
where $k$ is some non-negative integer.
The theorem then follows from a case by case analysis using Proposition \ref{proirr}.  We will do the $A_3$ and $A_4$ cases; the remaining cases are similar. ($A_3$) In this case, according to 
Proposition \ref{proirr} and Beauville \cite{b} Theorem 5.4,  $P(\widetilde{N},N)$ is the product of a dimension three abelian variety and the Jacobian of the tail arising from stable reduction, which is an elliptic curve.  The dimension of the locus of ppavs arising as a limit in this way is thus at most seven.  ($A_4$)  In this case, Proposition \ref{proirr} and Beauville \cite{b} Theorem 5.4 imply that the Prym is the product of the Jacobian of the genus two hyperelliptic tail, with a Jacobian of a genus three curve.  By Lemma \ref{lemdec2} it is in fact the product of hyperelliptic Jacobians.  The dimension of the locus of ppavs arising as a limit in this way is thus at most eight.
\end{proof}

The contraction of the divisors in the theorem  can also be understood  in the following way.  These divisors  correspond to one parameter  families of cubics degenerating to a singular cubic $X_0$, which we can also view as one parameter families of coverings of plane quintics.  The limit intermediate Jacobian  has two parts: the Prym of the covering obtained from normalizing the double cover of the plane quintic corresponding to $X_0$, which we will denote $N\widetilde{D}\to ND$, and the Jacobian of the tail, denoted $T$.  Although it is only proven in the cases of $A_1$ and $A_2$ singularities, we expect that $P(N\widetilde{D},ND)\cong J\hat{C}$, where $\hat{C}$ is the normalization of the $(2,3)$ curve $C\subset\mathbb P^4$ corresponding to the singular cubic $X_0$ (Theorem \ref{teojac} and Remark \ref{colmurconj}).
In fact, for all $A_k$ singularities, the covering $N\widetilde{D}\to ND$ is connected, and as pointed out in Remark \ref{colmurconj}, the associated Prym is a  Jacobian of a curve of the same type as $\hat{C}$. In particular, the dimension of the locus of such Pryms is the same as the dimension of the locus of such Jacobians.  

The contraction of the divisors in the morphism $\widetilde{\mathcal  M}\to \bar{I}$ corresponds to the loss of the data of the points on the curve $\hat{C}$ identified under the normalization map, and the data of the points on the tail used to attach the tail $T$ to $ND$.  We illustrate this with the case of an $A_4$ singularity; the other cases are similar, and are worked out in the table below.  The cubic with an $A_4$ singularity corresponds to a $(2,3)$ curve $C\subset \mathbb P^3$ lying on a smooth cubic, and on a quadric cone (Lemma \ref{lemtype}).  Moreover, $C$ has a cusp at the vertex of the quadric, which implies that it is hyperelliptic.  We have seen in Corollary \ref{cortype} that the moduli of such cubics is the same as 
$\mathcal M^h_{3,1}$, the moduli of smooth genus three hyperelliptic curves with a marked point.  This has dimension six.  According to Hassett \cite{has}, the tail is a hyperelliptic curve of genus two, attached to $ND$ at a Weierstra\ss\   point of $T$; these are parameterized by a finite cover of $\mathcal M_2$.  In total, this gives nine parameters, corresponding to the divisor $D_{A_4}$.  Forgetting the marked point on $\hat{C}$ leaves eight parameters, which is the dimension of the expected image of $D_{A_4}$ in $\bar{I}$, namely $\bar{J}^h_{2,3}$.  The following table illustrates this correspondence  in the other cases\footnote{For the $D_4$ cases marked with a ($\star$) see Lemma \ref{lemtype}.}.

\begin{table}[htb]
\begin{center} \begin{tabular}{||c||c|c|c|c||c|c||} \hline $X\subset \mathbb P^4$ & $C_{(2,3)}\subset\mathbb P^3$&dim & Tail $T$ &dim & $JX$ &dim \\ \hline\hline secant&-- &-- & --& --&$J^h_5$&9\\ \hline\hline $A_1$&$\mathscr M_4$&9 & $\varnothing$& --&$\mathbb C^*$-ext $J_4$&9\\ \hline $A_2$&$\theta_{\textnormal{null}}\cap \mathscr M_4$ &8& $\mathscr M_{1,1}$&1& $K$&9\\ \hline $A_4$& $\mathscr M^h_{3,1}$ &5+1& $\mathscr M_2$&3& $J^h_{2,3}$&$\le 8$\\ \hline \hline $A_3$& $\mathscr M^h_{3,2}$&5+2& $\mathscr M_{1,2}$&1+1& $\mathbb C^*$-ext $J^h_{1,3}$& $\le 6$\\ \hline $A_5$& $\mathscr M_{2,2}$&3+2& $\mathscr M_{2,1}$&3+1& $\mathbb C^*$-ext  $J_{2,2}$ & $\le 6$\\ \hline $D_4$& $\star$ & $\star$ & $\mathscr M_{1,1}$&1& $\mathbb C^*$-ext $\star$&$\le$ 4\\  \hline 
\end{tabular}
\vspace{0.2cm}
\caption{The extended period map on the boundary divisors.}\label{tablecon}
\end{center}
\end{table}

\begin{rem}
In some situations it is possible to understand  the extension data in the table more precisely. For instance, in the case of a degeneration to a
cubic threefold with a unique $A_1$ singularity, where the limit of the family of abelian varieties is a $\mathbb C^*$-extension of the Jacobian of a non-hyperelliptic genus four curve $C$, it is a result of Collino \cite[p.463]{collinoirr} (cf. Debarre \cite{deb}) that the extension data corresponds to the difference of the two $g^1_3$'s on $C$ (in $JC/\pm1$).  
Understanding the extension data for all degenerations is closely related to giving a description of the boundary of $I$ in the
toroidal compactifications of $\mathcal A_5$ (see Remark \ref{remtor}).
\end{rem}

\section{The polarization classes for the compactifications of the moduli space of cubic threefolds}\label{sectcoef}
In the previous section we constructed a normal crossing compactification $\widetilde{\calM}$ of the moduli space of cubic threefolds which dominates  the three natural compactifications $\overline{\calM}$, $(\calB/\Gamma)^*$ and $\bar{I}$. Each of these spaces corresponds to a semi-ample line bundle $L_i$ on $\widetilde{\calM}$ for $i=0,1,2$, the pull-back of the natural polarizations on $\overline{\calM}$, $(\calB/\Gamma)^*$ and $\bar{I}$ respectively. We determine the class of these semi-ample line bundles,  and see that the three cases correspond to different weights given to the discriminant divisors.  As will be briefly discussed in \S\ref{secthl},  the situation is quite similar to that for low genus curves  (see the program of Hassett, as described in the work of Hassett-Hyeon \cite{hh}, as well as the explicit computation of Hyeon-Lee \cite{hl} for the genus $3$ case). 

\begin{notations}
To make clear on which space a divisor lives, we fix the notation as in table \ref{tablediv}. In particular, $D_{A_1}$ is the strict transform of the ``discriminant locus'' $\Sigma$, $\Sigma'$, $\Sigma''$, $H_{\delta}$, and $\bar{J}_4$. Similarly, $H$ is the strict transform of the ``hyperelliptic locus'' $H_{c}$, $H'$, $H''$, and $\bar{J}_5^h$. 
\end{notations}

\begin{table}[htb]
\begin{center}
\renewcommand{\arraystretch}{1.25}
\begin{tabular}{|c|c|c|l|l|}
\hline
Step&Space&Picard number&Divisors&Polarizations\\
\hline
3&$\widetilde{\calM}$&$\ge 7$&$D_{A_1}$, $H$ / &$L_0$, $L_1$, $L_2$\\ 
&&& $D_{A_5}$, $D_{D_4}$ /&\\
&&&$D_{A_2}$,  $D_{A_3}$, $D_{A_4}$&\\
\hline
2&$\overline{\calB/\Gamma}$&4&$\Sigma''$, $H''$ /  & $L_0''$, $L_1''$, --\\ 
&&&$D_{A_5}''$, $D_{D_4}''$&\\
\hline
1&$\widehat{\calM}$&2&$\Sigma'$, $H'$& $L_0'$, $L_1'$, --\\  
\hline\hline
\cite{A,yokoyama}&$\overline{\calM}$&1&$\Sigma$&$\Sigma$\\ 
\hline
\cite{act,ls}&$(\calB/\Gamma)^*$& 2& $H_{\Delta}$, $H_{c}$& divisor of an \\
&&&& autom. form \\ 
\hline
\cite{cg}& $\bar{I}$&$3$&$\bar{J}_4$, $\bar{J}_5^h$ / & $\lambda$\\
&&&-- , -- / &\\
&&&$K$&\\
\hline
\end{tabular} 
\vspace{0.2cm}
\caption{Conventions on divisors}\label{tablediv}
\end{center}
\end{table}

The main result of the section is the computation of the coefficients of $D_{A_1}$ and $H$ in the polarizations $L_0$, $L_1$,  $L_2$, as well as in the canonical class $K_{\widetilde{\calM}}$: 
 
\begin{teo}\label{teocoef}
With notations as above, in $\Pic(\widetilde{\calM})_\bQ$ modulo the exceptional divisors of $\widetilde{\calM}\to \widehat{\calM}$, we have $K_{\widetilde{\calM}}\equiv -\frac{7}{16}D_{A_1}+\frac{19}{8}H$ and:

\begin{eqnarray*}
L_0&\equiv&D_{A_1}+22 H \sim  K_{\widetilde{\calM}}+\frac{6}{11}D_{A_1},\\
L_1&\equiv&D_{A_1}+14 H\sim K_{\widetilde{\calM}}+\frac{17}{28}D_{A_1},\\
L_2&\equiv&\frac{1}{8}(D_{A_1}+2H)\sim K_{\widetilde{\calM}}+\frac{13}{8}D_{A_1},
\end{eqnarray*} 
where $D\sim D'$ indicates that $D\equiv \alpha D'$ for some $\alpha\in \mathbb Q$.
\end{teo}
\begin{proof}
This is Corollary \ref{computek} for the canonical class, and Proposition \ref{computel0}, Corollary \ref{computel1}, and Proposition \ref{computel2} for $L_0$, $L_1$ and $L_2$ respectively.
\end{proof}
\subsection{The polarization of the GIT quotient $\overline{\calM}$}\label{space0}  The space $\overline{\calM}$ is a GIT quotient.  We have 
$$\overline{\calM}=\Proj \oplus_{k\ge 0} H^0(\bP^N, \calO_{\bP^N}(k))^G$$
with the tautological polarizations $\calO_{\overline{\calM}}(1)$ descending from $\calO_{\bP^N}(1)$. It follows that the discriminant divisor $\Sigma$ is an ample divisor on $\overline{\calM}$ and it is proportional to the anti-canonical divisor. 

\begin{lem}\label{coefs}
In $\Pic(\overline{\calM})_\bQ$ the following relation holds
\begin{equation}
K_{\overline{\calM}}=-\frac{7}{16} \Sigma.
\end{equation}
\end{lem}
\begin{proof}
 On the projective space $\bP^N$ parametrizing cubics in $\bP^4$, we have that the degree of $K_{\bP^N}$ is $-(N+1)=-\binom{n+d+1}{d}$ and that of the discriminant divisor $\widetilde{\Sigma}$ is $(n+2)(d-1)^{n+1}$, where $n=3$ and $d=3$ are the dimension and the degree respectively. Thus, $K_{\bP^N}=-\frac{7}{16} \widetilde{\Sigma}$. Since the stable cubic threefolds with trivial automorphism group form an open subset with complement of codimension at least two  in $\bP^N$, the previous relation descends to the GIT quotient $\overline{\calM}$.\end{proof}

The line bundle $L_0$ on $\widetilde{\calM}$ is defined as the pull-back of $\Sigma$ via the morphism $f:\widetilde{\calM}\to \overline{\calM}$ (see Diagram \ref{resdiag}), and we are interested in computing the coefficient of $D_{A_1}$ and $H$ in $L_0$.  From the description of the construction of $\widetilde{\mathcal M}$, this is equivalent to computing the coefficients of $\Sigma'$ and $H'$ in $L_0'$ at the intermediate level $\widehat{\calM}$. We obtain:

\begin{pro}\label{computel0} 
The divisor $L_0':=\epsilon_1^* \Sigma$ that gives the morphism to $\widehat{\calM}\to \overline{\calM}$ is
\begin{equation}
L_0= \Sigma' +22 H'.
\end{equation}
\end{pro}
\begin{proof}
The morphism $\widehat{\calM}\to \overline{\calM}$ is obtained by applying the first step of the Kirwan  desingularization procedure \cite{kirwan} (compare \cite[\S4]{act}). Namely, in $\bP^N$ we blow-up the locus  $GZ_R^{ss}$ corresponding to  the chordal cubic. The connected component of the stabilizer is $R=\SL(2)$. The local structure of the quotient $\overline{\calM}$ near the blow-up point is described by Luna's slice theorem (cf. \cite[pg. 198]{GIT}), i.e. locally (in the analytic topology) $\overline{\calM}$ is the normal slice $\calN$ to the orbit modulo the stabilizer $\SL(2)$. The stabilizer $\SL(2)$ acts on the ambient vector space $\bC^{N+1}$ as on $\Sym^3 W$ with $W=\Sym^4 V$, where $V$ is the standard representation of $\SL(2)$. Interpreting geometrically the irreducible components of the representation $\Sym^3(\Sym^4 V)$ one obtains that the action of $\SL(2)$ on $\calN$ is the irreducible representation $\Sym^{12} V$ (for similar computations see \cite[Ch. 11]{fh}). In conclusion, locally at the bad point $\overline{\calM}$ looks like the affine cone over $\bP(\Sym^{12} V)\gquot \SL(2)$.  Note also that locally the discriminant divisor is just the cone over the discriminant $\Sigma'$ in $\bP(\Sym^{12} V)\gquot \SL(2)$ (see \cite[\S4, \S5]{act} and \cite{smithvarley}). The resolution process blows-up the vertex of the cone. It follows that the coefficient of $H$ in $\epsilon_1^*\Sigma$ is the degree of $\Sigma'$, i.e. $22$.
\end{proof}

\subsection{Automorphic forms and the Baily--Borel compactification $(\calB/\Gamma)^*$}\label{space1} According to the Baily--Borel theory \cite{bb}, the divisor of an automorphic form on the ball $\calB$, with respect to the group $\Gamma$, is an ample divisor on $(\calB/\Gamma)^*$. Thus, if we are able to produce an automorphic form with known vanishing locus, the pull-back of its divisor gives the semi-ample line bundle $L_2$ on $\widetilde{\calM}$ defining the morphism $\widetilde{\calM}\to (\calB/\Gamma)^*$. In our situation, using a construction of Borcherds \cite{borcherds,bkpsb}, it is straightforward to produce such an automorphic form.

\begin{pro}\label{computealphabeta}
There exists an automorphic form of weight $w=48$ on $\calB$ with respect to the group $\Gamma$ vanishing precisely along the divisors $H_\Delta$ and $H_c$ with order of vanishing $\alpha=3$ and $\beta= 84$ respectively. 
\end{pro}
\begin{proof}
To construct the automorphic form $\phi$ we recall that the complex ball description of the moduli space of cubic threefolds is obtained by embedding the moduli space of cubic threefolds into the moduli space of cubic fourfolds (see \cite{act,ls}). In terms of hermitian symmetric domains this corresponds to the embedding $\calB\hookrightarrow \calD$ of the $10$-dimensional ball in a $20$-dimensional type IV domain $\calD$ (induced by the natural embedding of groups $U(1,10)\subset O(2,20)$). We construct $\phi$ by restricting to $\calB$ an automorphic form $\widetilde{\phi}$ on $\calD$ (for a similar construction see \cite{af}). 

The automorphic form $\widetilde{\phi}$ on $\calD$ is obtained by restricting the automorphic form $\Phi_{12}$ constructed by  Borcherds in \cite{borcherds}; this is similar to the case of degree two $K3$ surfaces discussed in Borcherds et al. \cite[Thm. 1.2, Ex. 2.1]{bkpsb}. We sketch the argument. The primitive cohomology of  a cubic fourfold is isometric to the lattice $\Lambda=A_2\oplus E_8^{\oplus 2}\oplus U^{\oplus2}$ which has a primitive embedding into the unimodular lattice $II_{26,2}$ with orthogonal complement the lattice $E_6$. On the lattice $II_{26,2}$, Borcherds \cite[Ex. 2 on pg. 195]{borcherds} has constructed an automorphic form $\Phi_{12}$ vanishing  with order $1$ along the hyperplanes determined by the roots of  $II_{26,2}$. Restricting $\Phi_{12}$ to $\calD$ and dividing by the linear forms corresponding to hyperplanes containing $\calD$ produces an automorphic form $\widetilde{\phi}$ vanishing precisely along the discriminant locus in $\calD$. In particular the weight of $\widetilde{\phi}$ is the weight of $\Phi_{12}$  plus half the number of roots of $E_6$, i.e. the weight of $\widetilde{\phi}$ is $48$.

It is easy to see that the two discriminant divisors $H_\Delta$ and $H_c$ are precisely the restrictions to $\calB$ of the discriminant divisors from $\calD$. Thus $\phi$, the restriction of $\widetilde{\phi}$ to $\calB$, vanishes exactly along  $H_\Delta$ and $H_c$. It remains to understand the order of vanishing. Note that $H_\Delta$ and $H_c$ correspond to codimension $2$ subspaces in $\calD$, or equivalently to embeddings of  lattices $L\hookrightarrow \Lambda$, where $L$ is a lattice of signature $(18,2)$. Dually, in terms of orthogonal complements in $II_{26,2}$, the two cases correspond to 
$E_6\hookrightarrow E_6\oplus A_2\hookrightarrow  II_{26,2}$ and $E_6\hookrightarrow E_8\hookrightarrow II_{26,2}$ respectively.  Since the automorphic form $\phi$ is obtained by restricting $\Phi_{12}$, the order of vanishing is  obtained by counting the number of roots. Concretely, it is half the difference between the number of roots in $E_6\oplus A_2$ (and $E_8$ respectively) and $E_6$. The claim follows.
\end{proof}

\begin{rem}
The same result can also be obtained by restricting the automorphic form constructed by Allcock \cite[Thm. 7.1]{allcockaut} on the $13$ dimensional complex ball.
\end{rem}

To compute the divisor of the automorphic form constructed in the previous proposition, one has to take into account the fact that $H_{\Delta}$ and $H_{c}$ are reflective divisors for $\Gamma$, i.e. ramification divisors for the morphism $\calB\to \calB/\Gamma$.
\begin{lem}\label{ramorder}
Let $\Gamma$ be the monodromy group associated to  the cubic threefolds acting on the $10$ dimensional complex ball $\calB$. Then, the ramification divisors of the natural projection $\calB\to \calB/\Gamma$ are  $H_\Delta$ and $H_c$. Furthermore, the group $\Gamma$ acts as a (complex) reflection in $H_\Delta$ and $H_c$ of order $3$ and $6$ respectively. 
\end{lem} 
\begin{proof}
This is \cite[Lemma 2.3]{act} for $H_\Delta$ and \cite[Lemma 4.3]{act} for $H_c$.
\end{proof}

We can now conclude our computation.
\begin{cor}\label{computel1}
The line bundle 
$$L_1':=\Sigma'+14H'$$
is a semi-ample line bundle on $\widehat{\calM}$ giving the morphism $\widehat{\calM}\to (\calB/\Gamma)^*$.
\end{cor}
\begin{proof}
As noted above, the divisor of an automorphic form is an ample divisor on $(\calB/\Gamma)^*$. By taking into account the ramification indices given by Lemma \ref{ramorder}, we obtain that the divisor of  the automorphic  form constructed in Proposition \ref{computealphabeta} is $H_{\Delta}+14H_{c}$  (see \cite[pg. 11]{af}). Since $\widehat{\calM}$ and $(\calB/\Gamma)^*$ differ in codimension larger than $2$ and  $H_\Delta$ and $H_{c}$ pull-back to $\Sigma$ and $H$ respectively, the claim follows. 
\end{proof}

In the case that $\Gamma$ is a neat group, it is well known that the automorphic forms are  pluricanonical sections. When $\Gamma$ is not neat, one has to add a correction factor for the reflective divisors of $\Gamma$ (due to the Riemann--Hurwitz ramification formula).  Concretely, in our situation we obtain the formula (see also  \cite[\S3]{alexeev}):
\begin{equation}
K_{(B/\Gamma)^*}+\left(1-\frac{1}{n_1}\right)H_\Delta+\left(1-\frac{1}{n_2}\right)H_c=\frac{d+1}{w}\mathrm{div}(\phi),
\end{equation}
where $n_1=3$ and $n_2=6$ are the ramification orders for the reflective divisors $H_\Delta$ and $H_c$ respectively, $d:=\dim \calB=10$, and $w$ is the weight of the automorphic form $\phi$.  As mentioned above, the correction factor $(1-\frac{1}{n_1})H_\Delta+(1-\frac{1}{n_2})H_c$ comes from the application of Hurwitz's formula. The computation of the scaling factor $\frac{d+1}{w}$ is standard (e.g. \cite[pg. 169]{l1}). We conclude:

\begin{cor}\label{computek}
With the notation as above,
$$K_{\widehat{\calM}}=-\frac{7}{16}\Sigma'+\frac{19}{8}H'.$$
\end{cor}

\subsection{Computation of the $\lambda$ class}\label{space2}
 On $\calA_5^*$ there exists a tautological polarization, the extension of the determinant of the Hodge bundle on $\mathcal  A_5$, which restricts to give a  polarization $\lambda$ on $\bar{I}$.
We set $L_2=\widetilde{\calP}_2^* \lambda$,  the pull-back of $\lambda$ to  $\widetilde{\mathcal M}$. Our aim is to express $L_2$ in terms of the  divisors $D_{A_1},\dots, H$.  
To do this it suffices to compute the intersection numbers $B.L_2$, $B.D_{A_1}$, etc. for a good choice of  test curves $B\subset \widetilde{\mathcal  M}$.  In fact, we will only determine the coefficients of $H$, $D_{A_1}$, and $D_{A_2}$.
Such a computation is possible because:
\begin{itemize}
\item Via projection, the computation of $B.L_2$ reduces to a computation of the degree of the $\lambda$ class which, for the families we choose, can be reduced to some standard computations.
\item The computation of $B.D_{A_1}$,  $B.D_{A_2}$, etc. reduces to the computation of the degree of the discriminant divisor, or to the structure of the discriminant near the bad singularities (i.e. the $A_2$ locus, etc.). 
\end{itemize} 

We prove the following:

\begin{pro}\label{computel2}
In $\Pic(\widetilde{\calM})_\bQ$ modulo the other exceptional divisors, we have 
$$
L_2\equiv \frac{1}{8}\left(D_{A_1}+2H+4D_{A_2}\right).
$$
     \end{pro}
\begin{proof}
This follows from the lemmas below.
\end{proof}
 
\subsubsection{Lefschetz pencils of cubic threefolds}
Let $\pi:\mathscr X \to B$ be a Lefschetz pencil of cubic threefolds; i.e. the general member of the family is  smooth  and each singular fiber has a unique $A_1$ singularity.  $B$ sits naturally in     $\overline{\mathcal{M}}$, but as the pencil is general and thus meets the discriminant transversally at smooth points, we may also view the family in $\widetilde{\mathcal{M}}$.  We would like to compute the intersection of $B$ with the various divisor classes on $\widetilde{\calM}$.

\begin{lem}\label{lemres}
In the notation above, 
$B.L_2=10$, $B.H=0$, $B.D_{A_1}=80$, $B.D_{A_2}=0$, and the intersection number of $B$ with the remaining boundary components is zero.
     \end{lem}
\begin{proof}
By push-pull, $f_*(B . f^*\Sigma)=B.\Sigma=80
$, the degree of $\Sigma$ (see for example the proof of Lemma  \ref{coefs}).  On the other hand, by construction, the intersection of 
$B$ with $H$ and each of the ``discriminant'' divisors on $\widetilde {\mathcal  M}$ other than $D_{A_1}$ is empty, and thus the intersection number is zero.  Since 
$f^*\Sigma=D_{A_1}+ \ldots$, it follows that $B.D_{A_1}=80$.

We now compute $L_2.B$.  Let $T$ be the subset of $B$ consisting of those points $t\in B$ such that $X_t$ is singular.  Set $B^\circ=B-T$, and $\mathscr X^\circ=\pi^{-1}B^\circ$.  Taking intermediate Jacobians induces a period map $B^\circ \to \mathcal  A_5$, and $L_2|_{B^\circ}$ is the pull-back of $\lambda$ via this map.  Observe also that $R^1\pi_*\Omega^2_{\mathscr   X^\circ/B^\circ}$ is the pull-back of $\Lambda$ from $\mathcal  A_5$.

For a smooth cubic threefold $X$ defined by a cubic polynomial $F$, the residue map induces an isomorphism 
$$
H^4(\mathbb P^4-X)\stackrel{res}{\to} H^3(X)
$$
so that under this isomorphism, forms in $H^{2,1}(X)$ can be represented by rational forms on $\mathbb P^4$ of the type
$$
\frac{A}{F^2}\sum_{\alpha}(-1)^\alpha z_\alpha\cdot dz_0\wedge \ldots \widehat{dz_\alpha}\ldots \wedge dz_4,
$$
where $z_0,\ldots,z_4$ are the homogeneous coordinates on $\mathbb P^4$, and $A$ is a linear function in the $z_i$.
A basis of such forms will give a frame trivializing $R^1\pi_*\Omega^2_{\mathscr   X^\circ/B^\circ}$ over 
$B^\circ$.  The Picard-Lefschetz theorem implies that the monodromy of the pencil is unipotent at the points of $T$, and it follows that the extension of this frame over $B$ induces an extension of $R^1\pi_*\Omega^2_{\mathscr   X^\circ/B^\circ}$ to a vector bundle $E$ over $B$, as well as an extension of the period map $B\to \mathcal  A_5^*$ such that 
$\det(E)$ is the pull-back of $\lambda$.
From the description of the frame, it is clear that for the pencil, each element of the frame transforms like $\mathscr O_B(2)$ so that $E\cong \mathscr O_B(2)^{\oplus 5}$, and 
$L_2.B=\det(E)=10$.
\end{proof}

\begin{rem}
Below we give an alternate method of computation for $L_2.B$ using Prym varieties.  
\end{rem}

\subsubsection{Pencil of hyperelliptic curves}
For this test curve we start on the opposite side of the diagram so to speak; we take a general pencil of hyperelliptic curves.  To be precise, consider the double cover of $\mathbb P^1\times \mathbb P^1$ branched along a general curve of type $(2g+2,d)$ with  $d$ even, and view it as a family of genus $g$ hyperelliptic curves via the second projection.  It is well known (cf. Harris-Morrison 
\cite{harrismorrison} p.293) that the family is Lefschetz, that there are $2d(2g+1)$ members with a unique node, and $\lambda.\mathbb P^1=gd/2$.
For our purposes, we can take $g=5$, and $d=2$.  
This pencil induces a pencil of hyperelliptic Jacobians, and 
we take as the test curve the lift $\tilde B$ of this family to $\widetilde{\calM}$.

\begin{lem}
In the notation above, 
$\tilde B.L_2=5$, $\tilde B.H=-2$, $\tilde B.D_{A_1}=44$, $\tilde B.D_{A_2}=0$, and the intersection number of $\tilde B$ with the remaining boundary components is zero.
     \end{lem}
\begin{proof}
 By construction $\tilde{B}.D_{A_1}=2d(2g+1)$, the intersection with the other ``discriminant'' divisors is zero, and  $\tilde B.\lambda=gd/2$.  
To compute $\tilde B.H$, we use the fact that
$\tilde B.L_0=0$.  This is a consequence of the containment  $\tilde B\subset H$, since $H$ is contracted by the morphism
$\widetilde{\mathcal  M}\to \overline{\mathcal  M}$ induced by $L_0$.  
Since $L_0=D_{A_1}+22H+\ldots$, it follows that 
$B.H=-44/22$.
\end{proof}

\subsection{Computation of the coefficient of $D_{A_2}$}\label{secda2} In this section we sketch the argument for the computation of the coefficient of $D_{A_2}$ in the polarizations $L_0$, $L_1$, and $L_2$. In principle, the computation can be extended to all the exceptional divisors of $\widetilde{\calM}\to \widehat{\calM}$.

\subsubsection{The blow-up of $\widehat{\calM}$ along the $A_2$ locus} The computation of the contribution of the $D_{A_2}$ divisor to the divisors $L_0$, $L_1$, and $K_{\widetilde{\calM}}$ is based on the following two observations. First, the $A_2$ locus corresponds to a certain type of intersection of hyperplanes from $\calH_{\Delta}$ (of codimension $2$ in $\calB$). Thus, up to factoring by the action of a finite group, the multiplicity of the $D_{A_2}$ divisor is computed as the multiplicity of the exceptional divisor in the blow-up of an intersection locus of some hyperplanes. Secondly, we note that we only have to analyze the last blow-up (i.e. the actual blow-up of the $A_2$ locus) in the procedure described in \S\ref{secresolution} (the other blow-ups do not affect the generic point of the $A_2$ locus).

By the work of Allcock--Carlson--Toledo \cite{act}, locally, the $A_2$ locus is the quotient of the intersection of type $\left(\begin{array}{cc}
3&\theta\\
\bar{\theta}&3
\end{array}\right)$  of hyperplanes\footnote{The matrix given here defines a lattice over the Eisenstein integers, which can be primitively embedded  into the lattice $\Lambda$ for cubic threefolds (see \cite{act}).  A root is a norm $3$ vector. The orthogonal hyperplane to a root defines a hyperplane from  $\calH_{\Delta}$. The meaning of the type of the intersections of the hyperplanes from $\calH_\Delta$ is the obvious one. A fact that we use is that, for an intersection of type $\left(\begin{array}{cc}
3&\theta\\
\bar{\theta}&3
\end{array}\right)$, there are $4$ hyperplanes meeting in a codimension $2$ linear subspace.} from $\calH_{\Delta}$ by the local monodromy group $\Gamma_{loc}$.  Ignoring the other blow-ups, the construction of \S\ref{secresolution}  blows-up this intersection and then takes the quotient. In other words, we have the diagram:
\begin{equation}
\xymatrix{
\calB\ar@{->}[d]_{\pi}&\ar@{->}[l]_{\phi}\ar@{->}[d]^{\widetilde{\pi}}\widetilde{\calB}\\
\calB/\Gamma&\ar@{->}[l]_{\epsilon}\widetilde{\calB}/\Gamma
}
\end{equation}
where the vertical arrows mean that we are taking a quotient, and $\phi$ is the blow-up of the linear space corresponding to the intersection inside $\calB$. Thus, we can compute the contribution of the $D_{A_2}$ divisor to the polarization classes $L_0$ and $L_1$ provided that we understand the local monodromy group $\Gamma_{loc}$ around the $A_2$ locus. The following lemma was communicated to the second author by D. Allcock. 

\begin{lem}
The local monodromy group around the $A_2$ locus is $\Gamma_{loc}\cong \SL_2(\mathbb{F}_3)$ (isomorphic to a central $\bZ/2$-extension of the alternating group $A_4$).
\end{lem}
\begin{proof}[Sketch of the proof]
The local fundamental group $\pi_1^{loc}$ around the $A_2$ locus is the braid group $B_3$ on $3$ strands (cf. \cite[Lemma 2.4]{act}). Thus, it is generated by $2$ elements (loops around the discriminant) $\sigma_1$ and $\sigma_2$. It is known that the images of the generators $\sigma_1$ and $\sigma_2$ under the monodromy representation $\pi_1^{loc}\to \Gamma$ are reflections of order $3$ (cf. \cite[Lemma 2.3]{act}). It follows that we have a surjection $G\to \Gamma_{loc}$, where $G=B_n/\langle \sigma_1^3,\sigma_2^3\rangle$. One checks that $G\cong \SL_2(\mathbb{F}_3)$ (using the natural map $B_3\to \SL_2(\bZ)\to \SL_2(\mathbb{F}_3)$). The claim follows for geometric reasons.\end{proof}

We conclude:

\begin{pro}
With notation as above. Let $\epsilon:\widehat{\calM}'\to \widehat{\calM}$ be the blow-up of the $A_2$ locus. Then,
\begin{eqnarray*}
\epsilon^*\Sigma'&=&D_{A_1}+6D_{A_2},\\
\epsilon^*H'&=&H,\\
K_{\widehat{\calM}'}&=&\epsilon^*K_{\widehat{\calM}}.
\end{eqnarray*}
\end{pro}
\begin{proof}
The hyperelliptic locus $H'$ meets the $A_2$ locus transversally. Thus, there is no $D_{A_2}$ contribution to the pull-back. To compute the pull-back of $\Sigma'$ we use $\phi^*\pi^*\Sigma'=\widetilde{\pi}^*\epsilon^* \Sigma'$. We have $\epsilon^* \Sigma'=3(H_1+\dots H_4)$, where $H_i$ are the corresponding reflection hyperplanes (and $3$ is the order of the reflection). It follows then that the multiplicity of the exceptional divisor $E_{A_2}$ in   $\phi^*\pi^*\Sigma'$ is $12$. The claim follows by noticing that there exists an involution (corresponding to multiplication by $-1$) in $\Gamma_{loc}$ fixing $E_{A_2}$ pointwise (i.e. $\widetilde{\pi}^*D_{A_2}=2E_{A_2}$). Similarly, we have 
$$K_{\widetilde{\calB}}=\phi^*K_{\calB}+E_{A_2}.$$
By Riemann--Hurwitz, it follows that 
\begin{eqnarray*}
K_{\calB}&=&\pi^*(K_{\widehat{\calM}}+\frac{2}{3}\Sigma'+\frac{5}{6} H'),\\
K_{\widetilde{\calB}}&=&\widetilde{\pi}^*(K_{\widehat{\calM}'}+\frac{2}{3}\Sigma'+\frac{5}{6} H'+\frac{1}{2} D_{A_2}).
\end{eqnarray*}
We conclude $K_{\widehat{\calM}'}=\epsilon^* K_{\widehat{\calM}}$.
\end{proof}

As mentioned above, the coefficients of $D_{A_2}$  in the  pull-backs  of divisors from $\widehat{\calM}$ to  $\widetilde{\calM}$ are the same  as those in the pull-backs to $\widehat{\calM}'$. Combining with Theorem \ref{teocoef}, we obtain:
\begin{cor}\label{corcoef}
In $\Pic(\widetilde{\calM})_{\bQ}$ modulo the other exceptional divisors, we have:
\begin{eqnarray*}
K_{\widetilde{\calM}}&\equiv& -\frac{7}{16}D_{A_1} -\frac{6\cdot 7}{16}D_{A_2}+\frac{19}{8}H,\\
L_0&\equiv&D_{A_1}+6 D_{A_2}+22 H  \sim  K_{\widetilde{\calM}}+\frac{6}{11}(D_{A_1}+6D_{A_2}),\\
L_1&\equiv&D_{A_1}+6 D_{A_2}+14 H  \sim K_{\widetilde{\calM}}+\frac{17}{28}(D_{A_1}+6D_{A_2}).\\
\end{eqnarray*} 
\end{cor}

\subsubsection{Moduli of \'etale double covers of curves and the Prym Hodge class} \label{sechodge}

 Let $\mathcal R_g$ denote the moduli of \'etale double covers of smooth curves of genus $g$.  There are three natural period maps,
$\mathcal R_g\to J_{2g-1}$, $\mathcal R_g\to J_g$ and $\mathcal R_g\to {P}_{g-1}$ given by taking the Jacobian of the covering curve, the Jacobian of the base curve, and the Prym variety, respectively.  Pulling back the Hodge classes via these maps gives three line bundles on $\mathcal R_g$.  We compare these, and give an application to cubic threefolds. 

To be precise, the Prym map extends to give a  morphism $\mathcal{P}:\overline{\mathcal{R}}_g\to \mathcal{A}_{g-1}^*$  (Friedman-Smith \cite{fs}).  Define the Prym Hodge line bundle  $\lambda_\eta$ on $\overline{\mathcal{R}}_g$ to be the pull back of the Hodge line bundle  on 
$\mathcal{A}_{g-1}^*$ via the Prym map.  We can view this as the determinant of a vector bundle on $\overline{\mathcal{R}}_g$.
To this end, let $(\widetilde{\mathscr C}\to \mathscr C)\stackrel{f}{\to} B$ be a family of double covers in $\overline{\mathcal{R}}_g$.   
There is an inclusion of vector bundles,  $0\to \Lambda|_B\to \widetilde{\Lambda}|_B$
where $\Lambda$ (resp. $\widetilde{\Lambda}$) is the pull back of the Hodge bundle on $\bar{\mathscr M}_g$ (resp. $\bar{\mathscr M}_{2g-1}$) to $\overline{\mathcal{R}}_g$.  
The quotient defines a vector bundle on $B$, and since our family was arbitrary, also on $\overline{\mathcal  R}_g$. We  call this the Prym Hodge bundle, and denote it  by $\Lambda_{\eta}$.
The following proposition  explains the choice of notation.
For the proof, we will want to use some notation from Bernstein \cite{bernstein} pertaining to divisor classes on $\overline{\mathcal{R}}_g$.  In particular there is a divisor 
$\Delta_0^r$ whose general point 
corresponds to an admissible double cover of an irreducible curve of genus $g$ with exactly one node, such that the covering curve has exactly one node, and the branches are not exchanged by the involution.

\begin{pro}
In the notation above, $\det\Lambda_\eta= \lambda_{\eta}$.
     \end{pro}
\begin{proof}
This is clear away from $\Delta_0^r$, and we so only need to show the proposition at a general point of $\Delta_0^r$, which corresponds to a cover of an irreducible nodal curve.  The proposition follows in this case from the usual Picard-Lefschetz arguments (cf. \cite{fs} Proposition 1.5 and Proposition 1.8). 
\end{proof}

  For computations on the associated stack, we will use the convention of referring to the associated class as $\delta_0^r$.   Also, set $\lambda$ (resp. $\tilde{\lambda}$) on $\overline{\mathcal{R}}_g$ to be the determinant of $\Lambda$ (resp. $\widetilde{\Lambda}$).

\begin{lem}\label{lember}
In the notation above
$$
\lambda_\eta=\lambda-\frac{\delta_0^r}{4}.
$$
     \end{lem}
\begin{proof}  From the discussion above, we have
$\lambda_\eta=\tilde{\lambda}-\lambda$.
 Bernstein \cite{bernstein} Lemma 3.1.3 states that
 $\tilde{\lambda}=2\lambda-\delta_0^r/4$.
\end{proof}

\subsubsection{Generalized Lefschetz and the $\lambda$ class} 
Let's apply this to the case of a generalized Lefschetz pencil of cubic threefolds $\mathscr X \to B$.  We will consider a special Lefschetz pencil as in Lemma \ref{lema2cub} obtained in the following way.  Let $X_F$ (resp. $X_G$) be a general smooth (resp. $A_2$) cubic threefold containing a line $\ell$, and defined by a polynomial $F$ (resp. $G$).  Take the pencil $T$ given by $sF+tG$.

\begin{lem}
In the notation above, 
$T.L_2=9 \frac{5}{6}$, $T.H=0$, $T.D_{A_1}=78$, $T.D_{A_2}=\frac{1}{6}$, and the intersection number of $T$ with the remaining boundary components is zero.
     \end{lem}
\begin{proof}
We compute $L_2.T$.  The other coefficients are computed similarly to the Lefschetz case (Lemma \ref{lemres}). 
By Lemma \ref{lema2qui}, projection from $\ell$ induces a transverse generalized Lefschetz degree three pencil   
$$(\tilde{\mathscr D}\to \mathscr D) \to B$$ 
of double covers of plane quintics (of type $A_2$).   Abusing notation, and ignoring finite coverings needed for stable reduction, we have $L_2.T=\lambda_{\eta}(B)$.
On the other hand, we have
$$
\lambda_\eta(B)=\lambda(B)-\frac{\delta_0^r(B)}{4}, \ \ \  \textnormal{ and  } \ \ \ 
\delta_0^r(B)=\frac{\delta_0(B)-\delta_0^u(B)}{2}.
$$
A standard computation (see the lemma below) shows that $\lambda(B)=3\frac{(d-1)(d-2)}{2}-1/6=3(6)-1/6$, where $d=5$ is the degree.  We also have $\delta_0(B)=3(n+2)(d-1)^{n+1}-2=
3(3)(16)-2$, where $d=5$ is the degree, and $n=1$ is the dimension (three times the degree of the discriminant for plane quintics minus two for the cusp), and $\delta_0^u(B)=5(16)-2$ (the degree of the discriminant for cubic threefolds minus two for the cusp).
\end{proof}

\begin{rem}
Observe that a similar computation with a Lefschetz pencil recovers the computation of the lambda class in Lemma \ref{lemres}.  Moreover, similar computations with generalized Lefschetz pencils of type $A_k$ can be carried out this way as well.
\end{rem}

\begin{lem}
Let $\mathcal C'\to B'$ be a degree $n$ Lefschetz pencil of plane curves of degree $d$.  Let $\mathcal C\to B$ be a transverse degree $n$ generalized Lefschetz pencil of type $A_2$,  of plane curves of degree $d$.  Then 
$$
\lambda(B)=\lambda(B')-\frac{1}{6}=n\frac{(d-1)(d-2)}{2}-\frac{1}{6},
$$
where by $\lambda(B)$ we mean that if $f:T\to B$ is any finite cover giving a stable reduction, then $\lambda(B)=\frac{\lambda(T)}{\deg(f)}$.
\end{lem}
\begin{proof}
Let $f:T\to B$ be the degree six map of bases for the stable reduction of the family.  There is the standard relation, 
$$
\lambda(T)=\frac{\kappa(T)+\delta(T)}{12}.$$
A well known computation (see for example Harris-Morrison \cite[Exercise 3.166]{harrismorrison}) shows that
$\kappa(T)=6(\kappa(B')-\frac{1}{6})$, and $\delta(T)=6(\delta(B')-2+\frac{1}{6})$.  Thus
$$\lambda(T)=\frac{6(\kappa(B')-\frac{1}{6})
+6(\delta(B')-2+\frac{1}{6})}{12}
$$ 
$$
= 6 \frac{\kappa(B')+\delta(B')-2}{12},
$$
so that $\lambda(B)=\frac{\lambda(T)}{6}=\lambda(B')-\frac{1}{6}$.  The computation of $\lambda(B')$ is standard.  One can use a residue argument as in Lemma \ref{lemres} for example. (See also Harris-Morrison \cite[p. 131]{harrismorrison}).
\end{proof}

\section{Relations to the birational geometry of $\overline{\calM}_3$}\label{secthl}
As pointed out in the introduction, the  relation between the various compactifications of the moduli space of cubic threefolds is similar to the relation between the various compactifications of the moduli space of low genus curves. Specifically, we discuss here a parallel between the construction of section \ref{secresolution} and the birational geometry $\overline{\calM}_3$ of the moduli space of genus $3$ curves.

To start, we note that for genus $3$ curves there exist, as in the case of cubic threefolds, three natural projective varieties birational to $\overline{\calM}_3$:  the GIT compactification $\overline{\calM}^{GIT}_3$ obtained by viewing the smooth nonhyperelliptic curves as plane quartics, the ball quotient description $(\calB_6/\Gamma_6)^*$ of Kondo \cite{k1}, and the closure of the Jacobian locus $\overline{J}_3$ (N.B. the situation is  similar also for $g=4$ curves, e.g. \cite{k2}).  By the work of Artebani \cite{artebani} and Looijenga \cite{l1}, the relation between the GIT compactification and the ball quotient description is essentially the same as for cubic threefolds, i.e. there exists an analogue $\widehat{\calM}_3$ of the space $\widehat{\calM}$ (see \cite[\S7.2]{l1}). Also,  the construction given in \S\ref{secresolution} applies ad litteram in the situation of curves as well. One can see that the resulting space $\widetilde{\calM}_3$ dominates $\overline{\calM}_3$, resolving the birational map $\overline{\calM}_3^{GIT}\dashrightarrow \overline{\calM}_3$.

In the case of genus $3$ curves,  in the general context of searching for the canonical model of $\overline{\calM}_g$ (see Hassett--Hyeon \cite{hh}), one has a quite precise understanding of the birational geometry of $\overline{\calM}_3$ (see  Hyeon--Lee \cite{hl} and Rulla \cite{rulla}). It is shown in  Hyeon--Lee \cite{hl} that the birational models of $\overline{\calM}_3$  mentioned above are log canonical models for an appropriate choice of the boundary. Specifically, we have the following identification between our notations and those of  \cite{hl}:
\begin{itemize}
\item[i)] $\overline{\calM}_3^{GIT}=\overline{\calQ}=\overline{\calM}_3(\frac{17}{28})$;
\item[ii)] $\widehat{\calM}_3=\overline{\calM}_3^{hs}=\overline{\calM}_3(\frac{7}{10}-\epsilon)$;
\item[iii)] $(\calB_6/\Gamma_6)^*=\overline{\calM}_3^{cs}=\overline{\calM}_3(\frac{7}{10})$.
\end{itemize}
where 
$$\overline{\calM}_3(\alpha)=\Proj(\oplus_n H^0(\overline{\calM}_3, n(K_{\overline{\calM}_3}+\alpha \delta)).$$
It is also standard that $\overline{J}_3\cong\overline{\calM}_3(2)$.  There is an additional space occurring in \cite{hl}:  $\overline{\calM}_3^{ps}=\overline{\calM}_3(\frac{7}{10})$. With these notations,  we obtain the following diagram:
\begin{equation}\label{resdiagg3} \xymatrix{ &&&\widetilde{\calM}_3\ar@{->}[ld]_{\epsilon_3}\ar@{->}[rd]^{\epsilon}\ar@/_1pc/[lddd]_(0.6){\widetilde{\calP}_c}\ar@/^4pc/[rddd]^{\widetilde{\calP}}\ar@/_2.5pc/[lllddd]_{f}&\\ &&\overline{\calB_6/\Gamma_6}\ar@{->}[rd]^{\epsilon'} \ar@{->}[ld]_{\epsilon_2}&&\overline{\calM}_3\ar@{=>}[dd]\ar@{=>}[ld]\\ &\widehat{\calM}_3 \ar@{=>}[ld]_(0.4){\epsilon_1}\ar@{=>}[rd]^{\widehat{\calP}_c}\ar@{<--}[rr]^(0.4){\phi}&&\overline{\calM}_3^{ps}\ar@{=>}[ld]\\ \overline{\calM}_3^{GIT}\ar@{-->}[rr]&&(\calB_6/\Gamma_6)^* \ar@{-->}[rr] &&\bar{J}_3   &\\ &&\calM_3^{nh}\ar@{^{(}->}[llu] \ar@{^{(}->}[u]^{\calP_c} \ar@{^{(}->}[rru]^{\calP} } \end{equation}
The diagram (\ref{resdiagg3}) shows how the morphisms arising from our construction fit together with those from Hyeon--Lee (where $\calM_3^{nh}$ denotes the moduli space of smooth non-hyperelliptic genus $3$ curves). Essentially,  the effect of our construction is to resolve the flip $\phi$ (that occurs at $\frac{7}{10}$).

We close by noting that the statement of theorem \ref{teocoef} is essentially equivalent to the determination of the coefficient $\alpha$ (i.e. $\frac{17}{28}$, $\frac{7}{10}$ and $2$) for $\overline{\calM}_3^{GIT}$, $(\calB_6/\Gamma_6)^*$ and $\overline{J}_3$ respectively. Namely, the computations of the previous section can be easily adapted to the case of genus $3$ curves, and indeed one recovers the numbers $\frac{17}{28}$, $\frac{7}{10}$ and $2$.
To exemplify, we discuss the analogue of the computation of \S\ref{space1} for genus $3$ curves. Similar to the case of cubics, there exist two natural discriminant divisors $\calH_{n}$ and $\calH_h$ corresponding to nodal and hyperelliptic genus $3$ curves respectively (see Kondo \cite{k1}). Again, they are reflective divisors, both of order $4$. On the other hand, by a result of Kondo \cite{kp}, there exists an automorphic form of weight $18$ that vanishes to order $2$ and $10$ along $\calH_n$ and $\calH_h$ respectively. Thus, on $\widehat{\calM}_3$ we have: 
\begin{eqnarray*}
L_1'&=&\Sigma'+5H',\\
K_{\widehat{\calM}_3}&=& -\frac{5}{9}\Sigma'+\frac{2}{9}H',
\end{eqnarray*}
where $\Sigma'$ and $H'$ are the discriminant and hyperelliptic divisors. Keeping in mind that the computations of Hyeon-Lee \cite{hl}  are at the level of stacks (i.e. $H'=\frac{1}{2}h^{hs}$,  and there is a correction term $h^{hs}$ for the canonical class), one checks that these equations give the $\frac{7}{10}$ factor of \cite{hh}, as well as the  expression $h^{hs}=9\lambda^{hs}-\delta^{hs}$ for the hyperelliptic divisor in terms of the standard generators of the Picard group of $\widehat{\calM}_3=\overline{\calM}_3^{hs}$.

\bibliography{cubicbib}

\end{document}